\documentclass[a4paper]{amsart}

\title{Continuous logic and Borel equivalence relations}

\author{Andreas Hallb\"ack}
\address{
  Institut de Math{\'e}matiques de Jussieu--PRG \\
  Universit\'e de Paris \\
  75205 Paris \textsc{cedex} 13 \\
  France}

\author{Maciej Malicki}
\address{
  Institute of Mathematics \\
  Polish Academy of Sciences \\
  ul. Sniadeckich 8, 00--656 \\
  Warsaw\\
  Poland}
\email{mamalicki@gmail.com}

\author{Todor Tsankov}
\address{
  Institut Camille Jordan \\
  Universit\'e Claude Bernard Lyon 1 \\
  43, boulevard du 11 novembre 1918 \\
  69622 Villeurbanne \textsc{cedex} \\
  France
  -- and --
  Institut Universitaire de France}
\email{tsankov@math.univ-lyon1.fr}


\usepackage[T1]{fontenc}
\usepackage[osf]{mathpazo}
\usepackage{eucal} 

\usepackage{hyperref}
\usepackage{my-macros}
\usepackage{stmaryrd}

\usepackage{enumitem}
\setlist[enumerate,1]{label=(\roman*), font=\normalfont}

\usepackage[initials, shortalphabetic]{amsrefs}

\DeclareMathOperator{\Mod}{Mod}
\DeclarePairedDelimiter\oset{\llbracket}{\rrbracket}

\newcommand{\tSen}{\mathrm{S}^{\mathrm{en}}}
\newcommand{\DeltaL}{\Delta_\mathrm{L}}
\newcommand{\tauqf}{t_{\mathrm{qf}}}

\newcommand{\Loo}{{\cL_{\omega\omega}}}

\newcommand{\Cstar}{$\mathrm{C}^\ast$}

\begin{document}

\begin{abstract}
We study the complexity of isomorphism of classes of metric structures using methods from infinitary continuous logic. For Borel classes of locally compact structures, we prove that if the equivalence relation of isomorphism is potentially $\mathbf{\Sigma}^0_2$, then it is essentially countable. We also provide an equivalent model-theoretic condition that is easy to check in practice. This theorem is a common generalization of a result of Hjorth about pseudo-connected metric spaces and a result of Hjorth--Kechris about discrete structures. As a different application, we also give a new proof of Kechris's theorem that orbit equivalence relations of actions of Polish locally compact groups are essentially countable.
\end{abstract}

\maketitle

\section{Introduction}
\label{sec:introduction}

The notion of Borel reducibility of definable equivalence relations was introduced in the foundational paper of Friedman and Stanley~\cite{Friedman1989} and since then, it has become a central part of modern descriptive set theory. In \cite{Friedman1989}, the authors were interested in one specific kind of equivalence relations---isomorphism of countable structures---and this still remains one of the best studied facets of the general theory. This setting allows to use methods from descriptive set theory, Polish group dynamics, and infinitary logic and their interplay leads to a rich and detailed theory. It was further developed in the papers of Hjorth--Kechris~\cite{Hjorth1996} and Hjorth--Kechris--Louveau~\cite{Hjorth1998}, where the Borel orbit equivalence relations of the infinite symmetric group $S_\infty$ were studied in detail. We recommend the book of Gao~\cite{Gao2009a} as a general reference for the more basic results. 

Hjorth's work on turbulence \cite{Hjorth2000} and many papers by various authors following it showed that a large class of equivalence relations coming from analysis cannot be captured by isomorphism of countable structures. This fueled the research on general orbit equivalence relations of Polish groups, mostly using methods from dynamics and Baire category. In many cases, proofs were driven by intuition from discrete model theory but no appropriate model-theoretic framework was available to formalize these ideas and they were often translated to the language of dynamics. Two notable examples are the papers of Becker~\cite{Becker1998a} and Hjorth~\cite{Hjorth2008a} (see also Melleray~\cite{Melleray2010}).

However, with the development of continuous logic in recent years, it is now possible to use model theory directly in this more general setting. The work of Gao and Kechris~\cite{Gao2003} and Elliot, Farah, Paulsen, Rosendal, Toms, and Törnquist \cite{Elliott2013} showed that the class of equivalence relations reducible to isomorphism of metric structures is exactly the same as those reducible to an orbit equivalence relation of a Polish group action. A major difference with the discrete setting is that countable structures can be thought of all having the same universe (some fixed countable set) and then isomorphism is nothing but the orbit equivalence relation of the natural action of $S_\infty$. This is very convenient for applying both dynamical and model-theoretic methods. In the continuous setting, this is no longer possible and two approaches of encoding separable metric structures have emerged in the literature. The first is encoding them as closed substructures of an appropriately chosen universal and sufficiently homogeneous structure (for example the Urysohn metric space). This is the approach taken in \cite{Gao2003} and \cite{Elliott2013}, where the authors show that if one restricts the class of subspaces appropriately, one can still recover isomorphism as the orbit equivalence relation of the isometry group of the Urysohn space. However, the encoding for achieving this is often cumbersome and the result is somewhat difficult to work with. An alternative method, closer to the encoding of discrete structures, was used by Ben Yaacov, Doucha, Nies, and Tsankov in \cite{BenYaacov2017a}. It is based on considering the values of the continuous predicates on a countable dense subset of the structure and recovering the whole structure from this information by taking the completion. This encoding allows for many interesting topologies on the space of structures, given by fragments of $\Lomo$, as in the discrete setting. Its main disadvantage is that one loses the group action, even though some dynamical methods, most notably a version of the Vaught transform, are still available. In order to compare the two methods, the reader may consult \cite{BenYaacov2017a} and the paper of Coskey and Lupini \cite{Coskey2016}: each proves a version of the López-Escobar theorem for the respective encoding.

First-order finitary logic is usually not expressive enough for descriptive set theoretic applications. Because of Scott sentences and the López-Escobar theorem, the logic that is usually employed for the study of the isomorphism equivalence relation of discrete structures is $\Lomo$, which allows for countable conjunctions and disjunctions. A continuous $\Lomo$ logic was first studied by Ben Yaacov and Iovino in \cite{BenYaacov2009a} and a continuous logic version of Scott analysis was developed in \cite{BenYaacov2017a}, laying the foundations for descriptive set theoretic applications.

In the beginning of this paper, we further develop the model theory of continuous $\Lomo$ logic and most notably the topometric structure of the type spaces. Because of the lack of compactness, there are some additional difficulties when compared with the usual continuous logic setting. Then, given an $\Lomo$ fragment $\cF$ and an $\cF$-theory $T$, we define a topology on the space $\Mod(T)$ of separable models of $T$. In analogy with the discrete setting, we identify precisely when an isomorphism class of a model is $G_\delta$.
\begin{theorem}
	\label{th:intro:AtomicG_delta}
	Let $\cF$ be a fragment, $T$ be an $\cF$-theory and $\xi \in \Mod(T)$. Let $t_\cF$ denote the topology on $\Mod(T)$ given by the fragment $\cF$ and let $[\xi]$ denote the set of models in $\Mod(T)$ isomorphic to $\xi$. Then the following are equivalent:
    \begin{enumerate}
    \item $[\xi]$ is $\bPi^0_2(t_\cF)$;
    \item $[\xi]$ is $t_\cF$-comeager in $\cl[t_\cF]{[\xi]}$;
    \item $[\xi]$ is $t_\cF$-non-meager in $\cl[t_\cF]{[\xi]}$;
    \item The model encoded by $\xi$ is $\cF$-atomic.
    \end{enumerate}
\end{theorem}
Even though to our knowledge, this is the first paper studying the fragment topologies in the continuous setting, a related topology, the one generated by the atomic formulas, has already been considered, for example by Cúth, Dole\v{z}al, Doucha, and Kurka in \cite{Cuth2019p}. If the fragment $\cF$ is the one of finitary logic and the theory $T$ eliminates quantifiers, then the two topologies coincide and this allows us to recover some results of \cite{Cuth2019p} about Banach spaces from our general theorem. 

Next we turn to the study of the situation where the isomorphism equivalence relation on models of a theory $T$ has low Borel complexity. Recall that if $\Gamma$ is a pointclass, $X$ is a standard Borel space and $E$ is a Borel equivalence relation on $X$, then $E$ is called \df{potentially $\Gamma$} if there is a Polish topology $\tau$ on $X$ compatible with its Borel structure such that $E \sub X^2$ is $\Gamma$ in $\tau \times \tau$. $E$ is called \df{essentially countable} if it is Borel reducible to a Borel equivalence relation with countable classes. It is obvious that an essentially countable equivalence relation is potentially $\bSigma^0_2$. A somewhat surprising converse to this for orbit equivalence relations of $S_\infty$ is due to Hjorth and Kechris~\cite{Hjorth1996}. One possible proof goes through a third equivalent model-theoretic condition that is easily verified in practice: isomorphism on a class of countable structures is essentially countable iff there exists a fragment $\cF$ such that for each structure $M$ in the class, there exists a tuple $\bar a \in M^k$ such that $\Th_\cF(M, \bar a)$ is $\aleph_0$-categorical. This theorem easily implies, for example, that isomorphism (on a Borel class of) of finitely generated discrete structures is essentially countable.

If one wants to generalize the theorem of Hjorth and Kechris to the continuous setting, some care is needed. First, the result for $S_\infty$ as stated above simply fails for general Polish groups. A simple way to see this is to consider the Banach space $\ell_1$ as an $F_\sigma$ Polishable subgroup of $\R^\N$; then the orbit equivalence relation given by the translation action $\ell_1 \actson \R^\N$ is $F_\sigma$ but is not essentially countable \cite{Hjorth2000}*{Proposition~3.25}. The corollary about finitely generated structures also spectacularly fails in the continuous setting: by combining several results in the \Cstar-algebraic literature with a theorem of Sabok~\cite{Sabok2016}, one sees that isomorphism for singly generated \Cstar-algebras is universal for orbit equivalence relations of Polish group actions (cf. Remark~\ref{rem:one-gen-Cstar-algebras}).

However, a form of the Hjorth--Kechris theorem is still true if one restricts to isomorphism of locally compact structures. There is also an appropriate model theoretic condition which is easy to check in applications (and implies the Hjorth--Kechris one in the discrete setting). We call a type $p$ \df{rigid} if for any two realizations $(M, \bar a)$, $(N, \bar b)$ of $p$ in separable models $M$ and $N$, we must have that $M$ and $N$ are isomorphic. Note, however, that the isomorphism need not send $\bar a$ to $\bar b$: this is what makes this condition weaker than just saying that $p$ is  $\aleph_0$-categorical as a theory.

If $T$ is a theory, we denote by $\cong_T$ the equivalence relation of isomorphism of models of $T$. The following is our main theorem.
\begin{theorem}
  \label{th:intro:loc-cpct-HK}
  Let $T$ be a countable, $\Lomo$ theory such that all of its separable models are locally compact. Then the following are equivalent:
	\begin{enumerate}
		\item $\cong_T$ is potentially $\bSigma^0_2$;
		
		\item \label{intro:th:loc-compact-HK:2} There exists a fragment $\cF$ such that for every $\xi \in \Mod(T)$, there is $k \in \N$ such that the set
		\begin{equation*}
		\set{\bar a \in M^k : \tp_\cF \bar a \text{ is rigid}}
		\end{equation*}
		has non-empty interior in $M^k$;
		
		\item $\cong_T$ is essentially countable.
	\end{enumerate} 
\end{theorem}
This theorem has a number of corollaries. The notion of a \df{pseudo-connected} locally compact metric space was introduced by Gao and Kechris in \cite{Gao2003} in order to study the complexity of isometry of locally compact metric spaces. For example, connected locally compact spaces and proper metric spaces are pseudo-connected. It is easy to see that metric structures whose underlying metric space is pseudo-connected satisfy condition \ref{intro:th:loc-compact-HK:2} above (in fact, \emph{all} types realized in such structures are rigid). So isomorphism of pseudo-connected metric structures is an essentially countable equivalence relation. This recovers a theorem of Hjorth, previously conjectured by Gao and Kechris (see \cite{Gao2003}*{Theorem~7.1}), for the pure metric space case and in fact, we have used some of the ideas of his proof.

Another well known result is Kechris's theorem \cite{Kechris1992a} that orbit equivalence relations of actions of Polish locally compact groups are essentially countable. This is also an easy consequence of Theorem~\ref{th:intro:loc-cpct-HK} after an appropriate encoding (cf. Corollary~\ref{c:Kechris}).

We expect that the continuous infinitary logic framework we build will find further applications in descriptive set theory. In view of this, in Section~\ref{sec:fragm-cont-l}, we develop fairly carefully the theory of type spaces, giving three equivalent definitions for the metric on them used to define the topometric structure. Section~\ref{sec:omitt-types-atom} is devoted to a self-contained proof of the omitting types theorem for $\Lomo$ continuous logic (the theorem is originally due to Eagle~\cite{Eagle2014}). In Section~\ref{sec:topol-gener-fragm}, we define the Borel space of models of a theory and the Polish topologies on it given by fragments of $\Lomo$. Finally, Sections \ref{sec:isomorphism} and \ref{sec:locally-compact} contain the proofs of our main results and Section~\ref{sec:pseudo-connected-metric-struct} is devoted to applications.

\subsection*{Acknowledgments} We are grateful to Itaï Ben Yaacov and Michal Doucha for useful discussions and to Ward Henson for providing useful references concerning the model theory of Banach spaces. Research on this paper was partially supported by the ANR project AGRUME (ANR-17-CE40-0026) and the \emph{Investissements d'Avenir} program of Université de Lyon (ANR-16-IDEX-0005).


\section{Fragments of continuous $\Lomo$ logic and type spaces}
\label{sec:fragm-cont-l}

\subsection{$\Lomo$ logic}
\label{sec:lomo-logic}

We start by recalling the setting of $\Lomo$ continuous logic. We mostly follow \cite{BenYaacov2017a}, however, the exposition is simplified by the fact that we do not need to keep careful track of moduli of continuity. A \df{modulus of continuity} is a continuous function $\Delta \colon [0, \infty) \to [0, \infty)$ satisfying for all $r, s \in [0, \infty)$:
\begin{itemize}
\item $\Delta(0)=0$;
\item $\Delta(r) \leq \Delta(r+s) \leq \Delta(r)+\Delta(s)$.
\end{itemize}
Suppose that $\Delta$ is a modulus of continuity and that $(X, d_X)$ and $(Y,d_Y)$ are metric spaces. We say that a map $f\colon X \to Y$ \df{respects} $\Delta$ if
\[ d_Y(f(x_1),f(x_2)) \leq \Delta(d_X(x_1, x_2)) \quad \text{ for all } x_1, x_2 \in X.\]

A \df{signature} $L$ is a collection of predicate and function symbols and as is customary, we treat constants as $0$-ary functions. Throughout the paper, we assume that $L$ is countable. To each symbol $P$ are associated its \df{arity} $n_P$ and its \df{modulus of continuity} $\Delta_P$. In addition, if $P$ is a predicate symbol, we associate to it its \df{bound}, a compact interval $I_P \sub \R$ where it takes its values. In a model $M$, predicate symbols are interpreted as real-valued functions of the appropriate arity respecting the modulus of continuity and the bound; similarly for function symbols. There is always a special binary predicate for the metric, denoted by $d$. To make sense of the modulus of continuity for symbols of arity greater than $1$, we need to fix a metric on tuples of elements of the model. By convention, if $M$ is a model with metric $d$ and $k \in \N$, we equip $M^k$ with the metric given by
\begin{equation*}
  d(\bar a, \bar b) = \max_i d(a_i, b_i).
\end{equation*}

Terms and atomic formulas are defined in the usual way. More general formulas are recursively defined as explained below; it is important to keep in mind that to any formula are associated its modulus of continuity and bound that can be calculated from its constituents. One can build new formulas using:
\begin{description}
\item[Finitary connectives] If $\phi$ and $\psi$ are formulas and $r \in \Q$, then $\phi + \psi$, $r \phi$, and $\phi \vee \psi$ are again formulas. Here $\phi \vee \psi$ is interpreted as $\max(\phi, \psi)$ and we also define $\phi \wedge \psi \coloneqq -(-\phi \vee -\psi) = \min(\phi, \psi)$. The constant $\bOne$ is also a formula. By the lattice version of the Stone--Weierstrass theorem (see \cite{Jonge1977}*{Theorem~13.12}), these connectives suffice to approximate any continuous combination of formulas. 
  
\item[Quantifiers] If $\phi(x, \bar y)$ is a formula, then $\sup_x \phi$ and $\inf_x \phi$ are also formulas.
  
\item[Infinitary connectives] If $\set{\phi_n(\bar x) : n \in \N}$ are formulas with the same finite set of free variables $\bar x$ that \emph{respect a common continuity modulus and bound}, then $\bigvee_n \phi_n$ and $\bigwedge_n \phi_n$ are also formulas. The symbol $\bigvee$ is interpreted as a countable supremum and $\bigwedge$ is interpreted as a countable infimum. The condition that we impose ensures that the interpretations of these formulas are still bounded, uniformly continuous functions.
\end{description}

The interpretations of formulas in a structure $M$ are defined in the usual way. We emphasize again that the interpretation of each formula of arity $k$ is a uniformly continuous, bounded function $M^k \to \R$ whose modulus of continuity and bound can be calculated syntactically from the formula (and thus are the same for all models). If $\phi(\bar x)$ is a formula, we will denote by $\phi^M$ the interpretation of $\phi$ in $M$. If $\bar a \in M^k$, we will often write $\phi(\bar a)$ instead of $\phi^M(\bar a)$. A \df{sentence} is a formula with no free variables and a \df{theory} is a collection of conditions of the form $\phi = c$, where $\phi$ is a sentence and $c \in \R$. A condition $\phi = c$ is \df{satisfied} in a structure $M$ if $\phi^M = c$. A structure $M$ is a \df{model} of the theory $T$, denoted by $M \models T$, if all conditions in $T$ are satisfied in $M$.

\begin{defn}
  A \df{fragment} of $\Lomo(L)$ is a countable collection $\mathcal{F} \sub \Lomo(L)$ that contains all atomic formulas and is closed under finitary connectives, quantifiers, taking subformulas, and substitution of terms for variables.
\end{defn}

The smallest fragment is \df{the finitary fragment} $\cL_{\omega\omega}(L)$ that contains no infinitary formulas. If $\cF$ is a fragment and $T$ is a theory, we will say that $T$ is an $\cF$-theory if all sentences that appear in $T$ are in $\cF$.


\subsection{Type spaces}
\label{sec:type-spaces}

Let $\cF \sub \Lomo(L)$ be a fragment. The collection of $\cF$-formulas over a fixed (finite or infinite) tuple of variables $\bar x$ form a Riesz space $F_{\bar x}$ with the operations given by the finitary connectives defined above. If $T$ is an $\cF$-theory, we have a natural seminorm on this space given by:
\begin{equation}
  \label{eq:norm-form}
  \nm{\phi}_T = \sup \set{|\phi(\bar a)| : M \models T, \bar a \in M^k}
\end{equation}
The set $\set{\phi : \nm{\phi}_T = 0}$ is an ideal in $F_{\bar x}$ and the completion of the quotient of $F_{\bar x}$ by this ideal is an archimedean Banach lattice (with unit $\bOne$) that will be denoted by $F_{\bar x}(T)$. Then we can define the space of \df{approximately realizable types} as follows:
\begin{equation*}
  \widehat \tS_{\bar x}(T) = \set{p \in F_{\bar x}^* : p(\phi \vee \psi) = p(\phi) \vee p(\psi) \text{ for all } \phi, \psi \in F_{\bar x} \And p(\bOne) = 1 }.
\end{equation*}
$\widehat \tS_{\bar x}(T)$ is clearly closed in the weak$^*$ topology, and therefore a compact space. We will often write $\phi(p)$ or $\phi^p$ instead of $p(\phi)$.

The topology on $\widehat \tS_{\bar x}$ is given by pointwise convergence on formulas, i.e., basic open sets are of the form
\begin{equation*}
  \oset{\phi < r} \coloneqq \set{p \in \widehat \tS_{\bar x}(T) : \phi(p) < r},
\end{equation*}
where $r \in \R$ and $\phi$ is a formula (and dually, $\oset{\phi > r} = \oset{-\phi < -r}$). 

If $|\bar x| = n$, we will also write $\widehat \tS_n(T)$ for $\widehat \tS_{\bar x}(T)$. If $M \models T$ and $\bar a \in M^{\bar x}$, the \df{type of $\bar a$} is defined by
\begin{equation*}
  \phi(\tp \bar a) = \phi(\bar a) \quad \text{for all } \phi \in F_{\bar x}(T).
\end{equation*}
Sometimes we also write $\tp_\cF(\bar a)$ to specify the fragment if it is not understood from the context. The set $\tS_{\bar x}(T)$ of \df{realizable types} (or just \df{types}) is defined by
\begin{equation*}
  \tS_{\bar x}(T) = \set{\tp(\bar a) : M \text{ is a model of } T, \bar a \in M^{\bar x}}.
\end{equation*}
If $\cF$ is $\Loo$, then the compactness theorem tells us that every approximately realizable type is realizable, i.e., $\widehat \tS_{\bar x}(T) = \tS_{\bar x}(T)$. For more general fragments, this is usually not the case. A typical situation in which a type $p$ is not realizable occurs when for some infinitary formula $\Phi = \bigwedge_k \phi_k$, we have that $\Phi(p) < \inf_k \phi_k(p)$. Nonetheless, we still have the following.
\begin{lemma}
  \label{l:realizable-dense}
  The set $\tS_{\bar x}(T)$ is dense in $\widehat \tS_{\bar x}(T)$.
\end{lemma}
\begin{proof}
  Recall that the $\dotminus$ operation is defined by $x \dotminus y = (x - y) \vee 0$. Suppose that for some formula $\phi(\bar x)$ and $r \in \R$ the open set $\oset{\phi < r} \sub \widehat \tS_{\bar x}(T)$ is non-empty. In particular, there is $p \in \widehat \tS_{\bar x}(T)$ such that $\phi(p) < r$. Then
  \begin{equation*}
    p(r \dotminus \phi) = r \dotminus p(\phi) > 0,
  \end{equation*}
  which implies that $\nm{r \dotminus \phi}_T > 0$. By the definition \eqref{eq:norm-form} of $\nm{\cdot}_T$, this means that there is $M \models T$ and $\bar a \in M^{\bar x}$ such that $\phi(\bar a) < r$.
\end{proof}
We will see later in Proposition~\ref{p:Sn-dense-Gdelta} that $\tS_{\bar x}(T)$ is a $G_\delta$ set and therefore a Polish space.

Next we see that the Banach lattice of formulas $F_{\bar x}(T)$ is isomorphic to the lattice $C(\widehat \tS_{\bar x}(T))$ of real-valued, continuous functions on $\widehat \tS_{\bar x}(T)$ equipped with the $\sup$ norm. This is just a version of the Yosida representation theorem, see \cite{Jonge1977}*{Section~13}.
\begin{prop}
  \label{p:Gelfand-duality}
  The map $\Gamma \colon F_{\bar x}(T) \to C(\widehat \tS_{\bar x}(T))$ defined by
  \begin{equation*}
    \Gamma(\phi)(p) = p(\phi)
  \end{equation*}
  is an isometric isomorphism of Banach lattices.
\end{prop}
\begin{proof}
  It is clear that $\Gamma$ is a lattice homomorphism. To see that it is isometric, note that, using Lemma~\ref{l:realizable-dense},
  \begin{equation*}
    \begin{split}
      \nm{\Gamma(\phi)} &= \sup \set{|\phi(p)| : p \in \widehat \tS_{\bar x}(T)} \\
        &= \sup \set{|\phi(p)| : p \in \tS_{\bar x}(T)} \\
        &= \sup \set{|\phi(\bar a)| : M \models T, \bar a \in M^{\bar x}} = \nm{\phi}_T.
    \end{split}
  \end{equation*}
  This implies that $\Gamma$ is injective and that its image is closed. The image is dense by the Stone--Weierstrass theorem, so $\Gamma$ is also surjective.
\end{proof}

\begin{remark}
  \label{rem:types-and-constants}
  There is a subtle feature of continuous $\Lomo$ logic regarding types and constants that can sometimes be confusing. In classical $\Lomo$ logic, as well as in finitary continuous logic, if $\phi(c)$ is a formula containing a constant symbol $c$, then we can replace all occurrences of $c$ by a variable $x$ and still obtain a valid formula $\phi(x)$. In particular, a type in $\tS_1(T)$ is nothing but a complete theory in the language expanded by a constant symbol extending $T$. In continuous $\Lomo$ logic, this is no longer the case. For a simple example, consider the sentence
\begin{equation*}
  \bigvee_n (n \cdot \nm{c} \wedge 1)
\end{equation*}
in the language of Banach spaces. This sentence evaluates to $1$ if $c \neq 0$ and to $0$ if $c = 0$ in any Banach space. However, replacing $c$ by a variable yields an invalid formula because it does not respect the equicontinuity rule for infinitary connectives (and indeed, its interpretation would be discontinuous at $0$).

This is what allows to have Scott sentences for structures of the form $(M, a)$, where the orbit $\Aut(M) \cdot a$ is not closed. Note that if $b \notin \Aut(M) \cdot a$, then $(M, a) \ncong (M, b)$. However, if $\phi(x)$ is any $\Lomo$-formula and $b \in \cl{\Aut(M) \cdot a}$, then $\phi^M(a) = \phi^M(b)$. So it is impossible to distinguish $(M, a)$ and $(M, b)$ by a formula $\phi(x)$ in the language of $M$ but it is possible to distinguish them by a sentence (in an appropriately rich fragment) if the language is augmented by a constant symbol.
\end{remark}

An important feature of type spaces in continuous logic is that, in addition to the logic topology, they are also equipped with a metric, coming from the metric on the models, which, in general, defines a finer topology. We recall that a \df{compact topometric space} is a triple $(X, \tau, \dtp)$, where $X$ is a set, $\tau$ is a compact Hausdorff topology on $X$ and $\dtp$ is a metric on $X$ which is \df{$\tau$-lower semicontinuous} (i.e., the set $\set{(x_1, x_2) \in X^2 : \dtp(x_1, x_2) \leq r}$ is $\tau \times \tau$-closed for every $r \geq 0$) and such that the topology defined by $\dtp$ is finer than $\tau$. We refer the reader to Ben Yaacov~\cite{BenYaacov2013b} for the general theory of topometric spaces.

We equip the type spaces $\widehat \tS_{\bar x}(T)$ with a topometric structure as follows. The topology $\tau$ is the logic topology defined earlier: namely, pointwise convergence on formulas. We recall from \cite{BenYaacov2017a}*{Section~7} the metric $\dtp$ on $\widehat \tS_{\bar x}(T)$ defined by:
\begin{equation}
  \label{eq:dtp1}
    \dtp(p, q) \leq s \iff \forall \phi \in \cF \quad  \big( \inf_{\bar y} \big( (d(\bar x, \bar y) \dotminus s ) \vee |\phi(\bar y) - \phi(p)| \big) \big)^q = 0.
  \end{equation}
  This definition is somewhat cumbersome and in \cite{BenYaacov2017a} it is only verified that $\dtp$ is a metric on the set of realizable types. We will give an equivalent definition which is easier to handle in some situations and, in particular, is obviously symmetric. We define for $p, q \in \widehat \tS_{\bar x}(T)$:
  \begin{multline}
    \label{eq:dtp2}
    \dtp(p, q) < s \iff \forall U \ni p, V \ni q \text{ $\tau$-open} \ \exists M \models T, \bar a, \bar b \in M^{\bar x} \\
     \tp(\bar a) \in U \And \tp(\bar b) \in V \And d(\bar a, \bar b) < s.
  \end{multline}
Our first task is to reconcile the two definitions.
\begin{prop}
  \label{p:equiv-dfn-topometric}
  The metrics defined by \eqref{eq:dtp1} and \eqref{eq:dtp2} are equal and $(\widehat \tS_{\bar x}(T), \tau, \dtp)$ is a compact topometric space.
\end{prop}
\begin{proof}
  For the first two paragraphs of the proof, we denote the metric defined in \eqref{eq:dtp1} by $\dtp'$. First we check that $\dtp \leq \dtp'$. To that end, suppose that $\dtp'(p, q) \leq s$ and fix $\eps > 0$ in order to show that $\dtp(p, q) < s + \eps$. Let $U = \oset{\phi < r}$ and $V$ be given. By decreasing $\eps$, we may assume that $\phi(p) < r - \eps$. From the density of realizable types and \eqref{eq:dtp1}, we know that there exists $M \models T$ and $\bar a \in M^{\bar x}$ such that $\tp \bar a \in V$ and
  \begin{equation*}
     M \models \inf_{\bar y} \big( (d(\bar a, \bar y) \dotminus s ) \vee |\phi(\bar y) - \phi(p)| \big) < \eps, 
   \end{equation*}
   i.e., there exists $\bar b$ such that $d(\bar a, \bar b) < s + \eps$ and $|\phi(\bar b) - \phi(p)| < \eps$, showing that $\tp \bar b \in V$ and $\dtp(p, q) < s + \eps$, as required.

   Next we show that $\dtp' \leq \dtp$. Suppose that $\dtp(p, q) < s$ in order to show that $\dtp'(p, q) \leq s$. Let $\phi \in \cF$ and $\eps > 0$ be given. Denote by $\psi(\bar x)$ the formula on the right-hand side of \eqref{eq:dtp1}. Using \eqref{eq:dtp2}, find $M \models T$ and $\bar a, \bar b \in M^{\bar x}$ such that $d(\bar a, \bar b) < s$, $|\phi(\bar b) - \phi(p)| < \eps$, and $|\psi(\bar a) - \psi(q)| < \eps$. It is easy to see now that $\psi(\bar a) < \eps$, implying that $\psi(q) < 2 \eps$. As $\eps$ was arbitrary, this shows that $\psi(q) = 0$ as desired.

   Next we check that $\dtp$ is a metric. It is obvious that $\dtp(p, p) = 0$. Suppose next that $p \neq q$ in order to show that $\dtp(p, q) > 0$. Let $\phi$ be a formula such that $\phi(p) < 0$ and $\phi(q) > 1$. If $\dtp(p, q) = 0$, by \eqref{eq:dtp2}, for every $\eps$, there exist $\bar a, \bar b$ with $d(\bar a, \bar b) < \eps$ and $\phi(\bar a) < 0$, $\phi(\bar b) > 1$, contradicting the uniform continuity of $\phi$.

   That $\dtp$ is symmetric follows directly from \eqref{eq:dtp2}. Next we verify the triangle inequality. Suppose that $\dtp(p_1, p_2) < s_1$ and $\dtp(p_2, p_3) < s_3$ in order to show that $\dtp(p_1, p_3) \leq s_1 + s_3$. Let $U_1=\oset{\phi_1 < r} \ni p_1$ and $U_3 = \oset{\phi_3 < r} \ni p_3$ be given $\tau$-open sets and let $\eps$ be arbitrary such that $\phi_1(p_1), \phi_3(p_3) < r - \eps$. Let
   \begin{align*}
     \psi_1(\bar x) &=  \inf_{\bar y} \big( (d(\bar x, \bar y) \dotminus s_1 ) \vee |\phi_1(\bar y) - \phi_1(p_1)| \big), \\
     \psi_3(\bar x) &=  \inf_{\bar y} \big( (d(\bar x, \bar y) \dotminus s_3 ) \vee |\phi_3(\bar y) - \phi_3(p_3)| \big).
   \end{align*}
   From \eqref{eq:dtp1}, we know that $\psi_1(p_2)=\psi_3(p_2) = 0$. By Lemma~\ref{l:realizable-dense}, there exists a model $M$ and $\bar b \in M^{\bar x}$ such that $\psi_1(\bar b) < \eps$ and $\psi_2(\bar b) < \eps$. Then there exist $\bar a, \bar c \in M^{\bar x}$ such that $\phi_1(\bar a) < r$, $\phi_3(\bar c) < r$, $d(\bar a, \bar b) < s_1 + \eps$, $d(\bar b, \bar c) < s_3 + \eps$. By the triangle inequality in $M$, $d(\bar a, \bar c) < s_1 + s_3 + 2\eps$. Thus $\tp \bar a \in U$, $\tp \bar c \in V$ and by \eqref{eq:dtp2} and the fact that $\eps$ was arbitrary, we have that $\dtp(p_1,p_3) \leq s_1+s_3$.
   
   That $\dtp$ is $\tau$-lower semicontinuous follows directly from \eqref{eq:dtp1}. We finally check that the $\dtp$-topology refines $\tau$. Let $(p_i)_i$ be a net that $\dtp$-converges to $p$. We need to check that for every formula $\phi$, $\phi(p_i) \to \phi(p)$. Let $\eps > 0$ be given and let $\delta > 0$ be such that for all models $M$ and $\bar a, \bar b \in M^{\bar x}$, $d(\bar a, \bar b) < \delta \implies |\phi(\bar a) - \phi(\bar b)| < \eps$. We have that for all sufficiently large $i$, there exist $M$, $\bar a, \bar b \in M^{\bar x}$ such that $|\phi(p_i) - \phi(\bar a)| < \eps$, $|\phi(p) - \phi(\bar b)| < \eps$, and $d(\bar a, \bar b) < \delta$, implying that $|\phi(p_i) - \phi(p)| < 3 \eps$. This concludes the proof of the proposition.
 \end{proof}

   Note that $\dtp(p, q) = \infty$ iff there exists a \emph{sentence} $\phi$ such that $\phi^p \neq \phi^q$ (this can happen  when the theory $T$ is not complete). Another important property of the metric $\dtp$ that follows directly from the definition is that for all $M$ and all $\bar a, \bar b \in M^{\bar x}$, we have that $\dtp(\tp \bar a, \tp \bar b) \leq d(\bar a, \bar b)$.

   We will say that a formula $\phi(\bar x)$ is \df{$1$-Lipschitz} if its interpretation $\phi \colon M^{\bar x} \to \R$ is a $1$-Lipschitz function for any model $M$. Equivalently, $\phi$ respects the continuity modulus $\DeltaL$ defined by $\DeltaL(r) = r$. We denote by $\cF_1$ the collection of $1$-Lipschitz formulas in the fragment $\cF$. The following proposition gives yet another useful equivalent definition for $\dtp$. A similar formula for $\Loo$ fragment was proved by Ben Yaacov~\cite{BenYaacov2013b}.

\begin{prop}
  \label{p:dtp-alt-def}
  \begin{enumerate}
  \item \label{i:p:dtp:1} Let $\Delta$ be a continuity modulus and let $\phi(\bar x)$ be a formula. Then $\phi$ respects $\Delta$ as a formula iff $\phi$ respects $\Delta$ as a function $(\tS_{\bar x}(T), \dtp) \to \R$.
    
  \item \label{i:p:dtp:2} For all $p, q \in \widehat \tS_{\bar x}(T)$,
      \begin{equation*}
    \dtp_\cF(p, q) = \sup_{\phi \in \cF_1} |\phi(p) - \phi(q)|.
  \end{equation*}

  \end{enumerate}
\end{prop}
\begin{proof}
  \ref{i:p:dtp:1} For the $(\Leftarrow)$ direction, note that for all $M \models T$ and $\bar a, \bar b \in M^{\bar x}$, we have:
  \begin{equation*}
    |\phi(\bar a) - \phi(\bar b)| = |\phi(\tp \bar a) - \phi(\tp \bar b)| \leq \Delta(\dtp(\tp \bar a, \tp \bar b)) \leq
    \Delta(d(\bar a, \bar b)).
  \end{equation*}
  For the $(\Rightarrow)$ direction, fix two types $p, q \in \tS_{\bar x}(T)$. Let $\eps > 0$. Find a model $M \models T$ and $\bar a, \bar b \in M^{\bar x}$ such that $|\phi(p) - \phi(\bar a)| < \eps$, $|\phi(q) - \phi(\bar b)| < \eps$, and $d(\bar a, \bar b) < \dtp(p, q) + \eps$. Then
  \begin{equation*}
    |\phi(p) - \phi(q)| \leq |\phi(\bar a) - \phi(\bar b)| + 2 \eps \leq \Delta(d(\bar a, \bar b)) + 2 \eps \leq
    \Delta(\dtp(p, q) + \eps) + 2\eps.
  \end{equation*}
  Taking $\eps \to 0$, we obtain the result.

  \ref{i:p:dtp:2} If $\phi$ is a $1$-Lipschitz formula, it follows from \ref{i:p:dtp:1} that $|\phi(p) - \phi(q)| \leq \dtp(p, q)$.
  
  Next, suppose that $\dtp(p, q) > s$. By \eqref{eq:dtp1}, there exists a formula $\phi(\bar y)$ such that denoting
  \begin{equation*}
    \theta(\bar x) = \inf_{\bar y}\ (d(\bar x, \bar y) \dotminus s) \vee |\phi(\bar y) - \phi(p)|,
  \end{equation*}
  we have $\theta(q) = r > 0$. Let
  \begin{equation*}
    \psi(\bar x) = \inf_{\bar y} \ d(\bar x, \bar y) \vee (s/r) |\phi(\bar y) - \phi(p)|
  \end{equation*}
  and note that $\psi(\bar x)$ is $1$-Lipschitz. Note also that if $p'$ is a realizable type, then $\psi(p') \leq (s/r)|\phi(p) - \phi(p')|$ (by taking $\bar y = \bar x$ in the $\inf$); taking a net of realizable types $p'$ converging to $p$ yields that $\psi(p) = 0$. On the other hand, we will check that $\psi(q) \geq s$. Let $\eps > 0$ and let $M \models T$ and $\bar a \in M^{\bar x}$ be such that $|\psi(q) - \psi(\bar a)| < \eps$ and $\theta(\bar a) > r - \eps$. Then we see that
  \begin{equation*}
    \psi(q) \geq \psi(a) - \eps \geq \frac{s(r - \eps)}{r} - \eps
  \end{equation*}
and letting $\eps \to 0$, we obtain $\psi(q) \geq s$.
\end{proof}


\section{Omitting types and atomic models}
\label{sec:omitt-types-atom}

\subsection{Isolated types and atomic models}
\label{sec:isol-types-atom}

If $p \in \tS_n(T)$ and $\delta > 0$, we will denote by $B_\delta(p)$ the open $\dtp$-ball around $p$ of radius $\delta$. A type $p \in \tS_n(T)$ is called \df{isolated} if it belongs to the $\tau$-interior of $B_\delta(p)$ for every $\delta > 0$ (or, in other words, if $\tau$ and the $\dtp$-topology coincide at $p$). This is equivalent to the formally weaker condition that $B_\delta(p)$ has non-empty $\tau$-interior for every $\delta$ \cite{BenYaacov2017a}*{Lemma~7.4}. A model $M$ is called \df{atomic} if for every $n$, all $n$-types realized in $M$ are isolated. We have the following basic lemma.
\begin{lemma}
  \label{l:compl-isolated-types}
  The set of isolated types in $\tS_n(T)$ is $\dtp$-closed and $\tau$-$G_\delta$.
\end{lemma}
\begin{proof}
  Let $I_n = \set{p \in \tS_n(T) : p \text{ is isolated}}$. First it is clear that $I_n$ is $G_\delta$ in $\tau$ because by definition, $I_n$ is exactly the set of points of continuity of the identity map $(\tS_n(T), \tau) \to (\tS_n(T), \dtp)$ (see, e.g., \cite{Kechris1995}*{Proposition~3.6}).

  Suppose now that $p_k \xrightarrow{\dtp} p$ and that $p_k \in I_n$ for all $k$. Let $\eps > 0$. Then there is $k$ and $\delta$ such that $B_\delta(p_k) \sub B_\eps(p)$. But by hypothesis, $B_\delta(p_k)$ has non-empty $\tau$-interior, and therefore, so does $B_\eps(p)$. By the remark above, this is sufficient to conclude that $p$ is isolated.
\end{proof}

An important property of atomic models is their uniqueness. The following standard fact is proved by the usual back and forth argument.
\begin{prop}
  \label{p:uniqueness-atomic}
  Let $M$ and $N$ be separable $L$-structures and suppose that there is a fragment $\cF$ such that $M \equiv_\cF N$ and $M$ and $N$ are $\cF$-atomic. Then $M \cong N$. \qed
\end{prop}

\subsection{Omitting types}
\label{sec:omitting-types}

The omitting types theorem is a fundamental tool in model theory and one of the few that do not depend on compactness. The version for classical $\Lomo$ logic is well known. In the continuous setting, the theorem (with a somewhat different formulation) is due to Eagle~\cite{Eagle2014}. The statement below in the case for finitary continuous logic is due to Ben Yaacov. We have preferred to include the proof as we think it is shorter and easier to follow than the one in \cite{Eagle2014}. We also use some of the constructions in defining the fragment topologies in the next section.

We fix a fragment $\cF$ and a countable $\cF$-theory $T$. If $\Xi \sub \tS_{\bar x}(T)$, we will say that a model $M \models T$ \df{omits} $\Xi$ if no type in $\Xi$ is realized in $M$.

\begin{theorem}[Omitting types]
    \label{th:omitting-types}
    Let $\cF$ be a fragment and let $T$ be an $\mathcal{F}$-theory. Suppose that for every $n$, we are given a $\tau$-meager and $\dtp$-open set $\Xi_n \sub \tS_{n}(T)$. Then there is a separable model $M \models T$ that omits all of the $\Xi_n$.
\end{theorem}

Throughout this subsection we fix a fragment $\cF$ and a theory $T$ as in the theorem. We will write $\prec$ to denote elementary substructures with respect to $\cF$.

The proof of the theorem depends on two lemmas. To state the first of them we need to define a convenient subspace of $\tS_\omega(T)$. This is the space of the types each of whose realizations enumerates a countable, dense subset of an elementary substructure. One can easily show that this is an intrinsic property of the type. More precisely, let $\bar x = (x_0, x_1, \ldots)$ be a countably infinite tuple of variables and, for $p \in \tS_{\bar x} (T)$, define
\begin{equation}
  \label{eq:Sen}
  p \in \tSen_{\bar x}(T) \iff \forall \phi \in \cF \quad \big(\inf_y \phi(y, \bar x) \big)^p = \inf_j \phi^p(x_j, \bar x).
\end{equation}
We will say that a type $p \in \tS_{\bar x}(T)$ \df{enumerates a model} if $p \in \tSen_{\bar x}(T)$. Next we have the following simple lemma that justifies the name.
\begin{lemma}
  \label{l:enum-model}
  Let $M \models T$, $\bar a \in M^\omega$, and let $p = \tp \bar a$. Then
  \begin{equation*}
  p \in \tSen_{\bar x}(T) \iff \cl{\set{a_i : i \in \omega}} \prec M.    
  \end{equation*}
\end{lemma}
\begin{proof}
  This is just a reformulation of the Tarski--Vaught test, see \cite{BenYaacov2008}*{4.5~Proposition}. In \cite{BenYaacov2008}, it is stated only for finitary continuous logic but the proof works equally well in the $\Lomo$ setting.
\end{proof}

\begin{lemma}
  \label{l:Sen-dense-Gdelta}
  $\tSen_{\bar x}(T)$ is a dense $G_\delta$ subset of $\widehat \tS_{\bar x}(T)$ in the topology $\tau$.
\end{lemma}
\begin{proof}
  First note that \eqref{eq:Sen} can be rewritten as: for all $\phi \in \cF$ and for all $r \in \Q$,
\begin{equation}
    \label{eq:Sen:1}
  \big(\inf_y \phi(y, \bar x) \big)^p < r \iff \exists j \ \phi^p(x_j, \bar x) < r.
\end{equation}
(The important direction is from left to right; the other is automatic for all $p$.) Let $U_{\phi, r}$ denote the $G_\delta$ subset of $\widehat \tS_{\bar x}(T)$ defined by \eqref{eq:Sen:1}.

If $\Phi = \bigwedge_k \phi_k$ is an infinitary formula in $\cF$, let
\begin{equation*}
  V_\Phi = \set{p \in \widehat \tS_{\bar x}(T) : \Phi(p) = \inf_k \phi_k(p)}.
\end{equation*}
It is clear that $V_\Phi$ can be written as: $p \in V_\Phi$ iff for all $r \in \Q$,
\begin{equation}
  \label{eq:Sen:infconn}
  \Phi(p) < r \implies \exists k \ \phi_k(p) < r,
\end{equation}
which shows that $V_\Phi$ is a $G_\delta$ set. It is also clear that all realizable types are in $V_\Phi$, so by Lemma~\ref{l:realizable-dense}, each $V_\Phi$ is dense. The proof of the lemma will be complete when we show that each of the sets $U_{\phi, r}$ is dense and that
\begin{equation}
  \label{eq:Sen:3}
  \tSen_{\bar x}(T) = \bigcap_{\phi, r} U_{\phi, r} \cap \bigcap_\Phi V_\Phi.
\end{equation}

First we check that each $U_{\phi, r}$ is dense. Let $\oset{\psi < s} \sub \widehat \tS_{\bar x}(T)$ be non-empty open. By Lemma~\ref{l:realizable-dense}, there exists a realizable $p$ with $\psi(p) < s$. Let $\bar a \models p$ in some model $M$. We may assume that $\big(\inf_y \phi(y, \bar x) \big)^p < r$. Then there exists $b \in M$ such that $\phi(b, \bar a) < r$. Let $n$ be larger than the indices of all variables that appear in $\psi$ and also larger than $i_0, \ldots, i_{k-1}$. Finally, define $p' = \tp(a_0, \ldots, a_{n-1}, b, a_{n+1}, \ldots)$. It is clear that $p' \in \oset{\psi < s} \cap U_{\phi, r}$.

We finally verify \eqref{eq:Sen:3}. The $\sub$ inclusion being clear, we check the other. Let $p \in \widehat \tS_{\bar x}(T)$ belong to the intersection on the right-hand side. Let $\bar a$ be a realization of the $\Loo$ part of $p$ in some model $N$ (this means that $\phi(\bar a) = \phi(p)$ for every $\phi \in \cL_{\omega\omega}$). Such a realization exists by the compactness theorem. Let $M = \cl{\set{a_i : i \in \N}}$; we will check that $\bar a$ realizes all of $p$ in $M$ (this will imply, in particular, that $M$ is a model of $T$). We proceed to prove by induction on formulas that for every formula $\phi \in \cF$, $\phi^{M}(\bar a) = \phi(p)$. This is true by construction for atomic formulas. The induction step for finitary connectives follows from the definition of the type space. Let now $\phi(\bar x) = \inf_y \psi(y, \bar x)$. Suppose first that $\phi(p) < r$. As $p \in U_{\psi, r}$, we have that there exists $j$ with $\psi^p(x_j, \bar x) < r$. By the induction hypothesis, $\psi^M(a_j, \bar a) < r$, and so, $\phi^M(\bar a) < r$. Conversely, suppose that $\phi^M(\bar a) < r$. Then there is $b \in M$ such that $\psi^M(b, \bar a) < r$; as $\set{a_i : i \in \N}$ is dense in $M$, this means that there is $j$ such that $\psi^M(a_j, \bar a) < r$, and using that $p \in U_{\psi, r}$, this gives that $\phi(p) < r$, concluding the argument for quantifiers. The induction step for infinitary connectives follows from \eqref{eq:Sen:infconn}. This concludes the proof of the lemma.
\end{proof}

For each $s\in [\N]^{<\N}$, i.e., for each finite sequence of distinct natural numbers, we define a map 
$\pi_s \colon \tSen_{\bar x}(T) \to \tS_{|s|}(T)$ by 
\[\big( \phi(x_0,\ldots , x_{|s|-1}) \big)^{\pi_s(p)} = \big(\phi(x_{s_0},\ldots ,x_{s_{|s|-1}})\big)^p. \]

\begin{lemma}\label{l:projection-open}
For each $s\in [\N]^{<\N}$ the map $\pi_s|_{\tSen_{\bar x}(T)}$ is an open and continuous surjection for the topology $\tau$.
\end{lemma}
\begin{proof}
  Continuity is clear, so we proceed to prove that $\pi$ is surjective and open. Let $p \in \tS_n(T)$; let $M \models T$ and $\bar a \in M^n$ be such that $\tp \bar a = p$. By the downward Löwenheim--Skolem theorem, we may assume that $M$ is separable. Let $\bar b$ be a dense sequence in $M$ with $b_i = a_{s_i}$ for $i < n$. Then $\tp \bar b \in \tSen_{\bar x}(T)$ and $\pi_s(\tp \bar b) = p$.

  To check that $\pi_s$ is open, let $\oset{\phi(\bar x) < r}$ be a basic open set in $\tSen_{\bar x}(T)$. Let $x_{i_0}, \ldots, x_{i_{k-1}}$ be all variables that appear in $\phi$ and are not among $x_{s_0}, \ldots, x_{s_{n-1}}$. We claim that
  \begin{equation*}
    \pi_s(\oset{\phi < r}) = \oset{\inf_{x_{i_0}, \ldots, x_{i_{k-1}}} \phi < r}.
  \end{equation*}
  We only check the inclusion $\supseteq$. Let $p$ belong to the right-hand side and let $\bar a$ be a realization of $p$ in a separable model $M$. Then there exist $b_0, \ldots, b_{k-1}$ in $M$ such that $\phi(\bar a, \bar b) < r$. Finally, we can complete $\bar a \bar b$ to a dense sequence $\bar c$ such that $\pi_s(\tp \bar c) = p$.
\end{proof}
Lemmas \ref{l:realizable-dense} and \ref{l:projection-open} together give us the following.
\begin{prop}
  \label{p:Sn-dense-Gdelta}
  For every $n$, the set of realizable types $\tS_n(T) \sub \widehat \tS_n(T)$ is dense $G_\delta$ and therefore a Polish space.
\end{prop}

We are finally ready to prove the omitting types theorem.
\begin{proof}[Proof of Theorem~\ref{th:omitting-types}]
  Let $\bar x$ be an infinite tuple of variables and consider the subset $A \sub \tSen_{\bar x}(T)$ defined by
  \begin{equation*}
    A = \bigcup_{s \in [\N]^{< \N}} \pi_s^{-1}(\Xi_{|s|}).
  \end{equation*}
  As the preimage of a meager set by an open map is meager, Lemma~\ref{l:projection-open} implies that $A$ is meager. As $\tSen_{\bar x}(T)$ is Polish, this implies that there is $p \in \tSen_{\bar x}(T) \sminus A$. Let $\bar a$ be a realization of $p$ and let $M = \cl{\set{a_i : i \in \N}}$. We claim that $M$ omits all of the $\Xi_n$. Suppose not; then there is $n$, some $q \in \Xi_n$, and $\bar b \in M^n$ such that $q = \tp \bar b \in \Xi_n$. As $\Xi_n$ is $\dtp$-open, there exists $\eps > 0$ such that $B_\eps(q) \sub \Xi_n$. As $\bar a$ is dense in $M$, there exist $s_0, \ldots, s_{n-1}$ such that $d(\bar b, (a_{s_0}, \ldots, a_{s_{n-1}})) < \eps$. Then $\tp (a_{s_0}, \ldots, a_{s_{n-1}}) \in \Xi_n$, contradicting the fact that $\pi_s(p) \notin \Xi_n$.
\end{proof}

\begin{remark}
  \label{rem:strong-omitting-types}
  We note that the proof above gives the following stronger version of Theorem~\ref{th:omitting-types} that will be used in the sequel. Namely, under the assumptions of the theorem, for comeagerly many $\xi \in \tSen_{\bar x}(T)$, the model coded by $\xi$ omits all of the $\Xi_n$.
\end{remark}


\section{Topologies generated by fragments}
\label{sec:topol-gener-fragm}

In the next sections, we are going to discuss the equivalence relation of isomorphism of separable models of a given theory $T$. In order to do this, we need to define a suitable Polish topology (or at least a standard Borel structure) on the set of these models. It is possible to do this in many different ways but the most convenient one for us will be to use the space $\tSen_{\bar x}(T)$ of types enumerating models defined in the previous section---it clearly codes all the separable models of $T$. It is worth noting at this point that this space and its standard Borel structure do not depend on the fragment $\cF$ that we have chosen; the topology, however, does: if the fragment has more formulas, then the topology has more open sets. By Lemma~\ref{l:Sen-dense-Gdelta}, this topology is Polish as long as the fragment $\cF$ contains all sentences in $T$. In order to avoid subscripts, it will also be convenient for us to identify the variable $x_i$ with the natural number $i$. With this convention, our setting is equivalent to the usual approach in descriptive set theory to consider structures defined on $\omega$. 
We will denote by $\Mod(T)$ the standard Borel space $\tSen_{\bar x}(T)$ and by $t_\cF$ the Polish topology on $\Mod(T)$ generated by the fragment $\cF$. 

If $\xi \in \Mod(T)$, we will denote by $M_\xi$ the model enumerated by $\xi$, i.e., the structure $\cl{\set{a_i : i \in \N}}$ for any realization $\bar a \models \xi$. We write $M \cong N$ if the models $M$ and $N$ are isomorphic and $\xi \cong \eta$ if $M_\xi \cong M_\eta$. We will also denote by $[\xi] = \set{\eta \in \Mod(T) : \xi \cong \eta}$ the isomorphism class of $\xi$. We will write $M \equiv_\cF N$ if the models $M$ and $N$ are elementarily equivalent with respect to $\cF$, that is, for all sentences $\phi \in \cF$, we have $\phi^M = \phi^N$.

\begin{prop}
	\label{p:ElementaryClosure}
	Let $T$ be an $\cF$-theory, $\xi, \eta \in \Mod(T)$. Then $M_\xi \equiv_\cF M_\eta$ if and only if $\overline{[\xi]}^{t_\cF}=\overline{[\eta]}^{t_\cF}$.
\end{prop}
\begin{proof}
  For any sentence $\phi \in \cF$ and $r \in \R$, the set $\set{\xi \in \Mod(T) : \phi^\xi = r}$ is invariant under isomorphism and closed in $\Mod(T)$, so the backward direction is clear.
  
  Assume now that $M_\xi \equiv_\cF M_\eta$. Fix $\phi(\bar{x}) \in \cF$, $r \in \R$, and $u \in \N^{\bar x}$ and suppose that $[\xi] \cap \oset{\phi(u) < r} \neq \emptyset$. Then
\begin{equation*}
  \big( \inf_{\bar{x}} \phi(\bar{x}) \big)^\eta = \big( \inf_{\bar{x}} \phi(\bar{x}) \big)^\xi < r,
\end{equation*}
  so there exists $\bar{b}$ in $M_\eta^{\bar x}$ such that $\phi^{M_\eta}(\bar{b}) < r$. But this means that there exists $\zeta \in [\eta]$ such that $\phi^{\zeta}(u) < r$, i.e., $[\eta] \cap \oset{\phi(u) < r} \neq \emptyset$. Thus, $\overline{[\xi]}^{t_\cF}=\overline{[\eta]}^{t_\cF}$.
\end{proof}

\begin{cor}
  \label{c:CategoricalClosed}
  For any $\xi \in \Mod(T)$, $\Th_\cF (M_\xi)$ is $\aleph_0$-categorical if and only if $[\xi]$ is closed in the topology $t_\cF$.
\end{cor}
\begin{proof}
  If $\Th_F(M_\xi)$ is  $\aleph_0$-categorical, then $[\xi] = \set{\eta : M_\eta \equiv_\cF M_\xi}$, which is a closed set. The converse follows from Proposition~\ref{p:ElementaryClosure}.
\end{proof}

\begin{theorem}
	\label{th:AtomicG_delta}
	Let $\cF$ be a fragment, let $T$ be an $\cF$-theory and let $\xi \in \Mod(T)$. Then the following are equivalent:
    \begin{enumerate}
    \item $[\xi]$ is $\bPi^0_2(t_\cF)$;
    \item $[\xi]$ is $t_\cF$-comeager in $\cl[t_\cF]{[\xi]}$;
    \item $[\xi]$ is $t_\cF$-non-meager in $\cl[t_\cF]{[\xi]}$;
    \item $M_\xi$ is $\cF$-atomic.
    \end{enumerate}
\end{theorem}
\begin{proof}
  \begin{cycprf}
    \item[\impnnext] This is clear.

    \item[\impnext] Suppose that $M_\xi$ is not atomic and let $T' = \Th_\cF(M_\xi)$. Then there exists $n \in \N$ and a type $p_0 \in \tS_n(T')$ realized in $M_\xi$ which is not isolated. By Theorem~\ref{th:omitting-types} and Remark~\ref{rem:strong-omitting-types}, for comeagerly many $\eta \in \Mod(T') = \cl[t_\cF]{[\xi]}$, $M_\eta$ omits $p_0$, implying that $\eta \ncong \xi$.
      
    \item[\impfirst] By the uniqueness of atomic models (Proposition~\ref{p:uniqueness-atomic}), a model $M$ is isomorphic to $M_\xi$ iff $M$ is an $\cF$-atomic model of $T' = \Th_\cF(M_\xi)$. Let $I_n \sub \tS_n(T')$ be the set of isolated types. It follows from Lemma~\ref{l:compl-isolated-types} that $I_n$ is $G_\delta$ and that for every $n$ and every model $M$, the set $\set{\bar a \in M^n : \tp \bar a \in I_n}$ is closed. Thus we have:
  \begin{equation*}
    \eta \in [\xi] \iff \forall u \in \N^{<\N} \quad \tp^{\eta}(u) \in I_{|u|},
  \end{equation*}
  which is clearly a $G_\delta$ condition.
  \end{cycprf}
\end{proof}

A different, coarser topology $\tauqf$ on the space of models $\Mod(T)$ often considered in the literature is the one generated by the atomic formulas (rather than all formulas in a certain fragment). Then in order to ensure that this topology is Polish, one usually restricts to $\forall \exists$-theories, i.e., theories axiomatized by conditions of the form
\begin{equation*}
  \qsup_{\bar x} \qinf_{\bar y} \phi(\bar x, \bar y) \leq 0
\end{equation*}
with $\phi$ a quantifier-free finitary formula. This topology is harder to handle theoretically because of the lack of quantifiers and its heavy dependence on the choice of signature but is easier to compute with in practice. Fortunately, in some common situations, the topology $\tauqf$ coincides with the topology $t_0$ generated by the fragment $\cL_{\omega\omega}(L)$: namely, when the theory $T$ is model-complete. Recall that an $\Loo$-theory is \df{model-complete} if every embedding  between models of $T$ is elementary. Equivalently, every formula is equivalent to a formula of the form $\inf_{\bar y} \psi(\bar x, \bar y)$ with $\psi$ quantifier-free. In particular, if a theory eliminates quantifiers, it is model-complete. Thus we have the following corollary of Theorem~\ref{th:AtomicG_delta}, which is a convenient way to show the existence of $G_\delta$ isomorphism classes in the topology $\tauqf$.

\begin{cor}
  \label{c:tau-qf}
  Let $T_0$ be a $\forall \exists$-theory and let $T \supseteq T_0$ be an $\cL_{\omega\omega}(L)$-theory. Let $\xi \in \Mod(T)$ and consider the statements:
  \begin{enumerate}
  \item \label{i:tauqf:1} $[\xi]$ is $G_\delta$ in $(\Mod(T_0), \tauqf)$;
  \item \label{i:tauqf:2} $M_\xi$ is an atomic model of its $\cL_{\omega\omega}$-theory.
  \end{enumerate}
  Then \ref{i:tauqf:1} $\Rightarrow$ \ref{i:tauqf:2} and if $T$ is model-complete, we have equivalence.
\end{cor}
\begin{proof}
  \begin{cycprf}
  \item[\impnext] As the topology $t_0$ is finer than $\tauqf$, we have that $[\xi]$ is a $G_\delta$ set in $t_0$, so we can apply Theorem~\ref{th:AtomicG_delta}.
    
  \item[\impfirst] We will show that the topologies $\tauqf$ and $t_0$ coincide on $\Mod(T)$. Let $\set{\xi : \phi^\xi(u) < r}$ be a basic open set in $t_0$, where $\phi$ is a $\Loo$ formula and $u \in \N^k$. By model-completeness of $T$, there exists a quantifier-free formula $\psi(\bar x, \bar y)$ such that
    \begin{equation*}
      \phi(\bar b) = \qinf_{\bar y} \psi(\bar b, \bar y), \quad \text{ for all } M \models T, \bar b \in M^k.
    \end{equation*}
    Thus, 
    \begin{equation*}
      \phi^\xi(u) < r \iff \exists v \in \N^k \ \psi^\xi(u, v) < r
    \end{equation*}
    and the latter is clearly an open condition on $\xi$ in $\tauqf$.

    Thus $(\Mod(T), \tauqf)$ is Polish and therefore a $G_\delta$ subset of $(\Mod(T_0), \tauqf)$. Now Theorem~\ref{th:AtomicG_delta} implies that $[\xi]$ is $G_\delta$ in $(\Mod(T), t_0)$ and thus Polish in both $t_0$ and $\tauqf$. Therefore $[\xi]$ is $G_\delta$ in $\Mod(T_0)$.
  \end{cycprf}
\end{proof}

As an illustration of how Corollary~\ref{c:tau-qf} applies, we explain how to recover some results from the recent paper \cite{Cuth2019p} of Cúth, Dole\v{z}al, Doucha, and Kurka. We consider the signature of Banach spaces with function symbols for addition and multiplication by scalars and a predicate symbol for the norm. Let $T_0$ be the theory of Banach spaces (this is a universal theory because a substructure of a Banach space in this signature is still a Banach space). Some examples of $\aleph_0$-categorical Banach spaces are the Gurarij space and $L^p([0, 1])$ for $1 \leq p < \infty$. Moreover, the Gurarij space and $L^p([0, 1])$ for $p$ not an even integer greater than $2$ eliminate quantifiers. The $\aleph_0$-categoricity and quantifier elimination for the Gurarij space follow from its homogeneity and the Ryll-Nardzewski theorem. For $L^p$, it follows from Henson~\cite{Henson1976} that $L^p$ is $\aleph_0$-categorical as a Banach lattice and it is again a consequence of the Ryll-Nardzewski theorem that a reduct of an $\aleph_0$-categorical structure is
$\aleph_0$-categorical. For quantifier elimination for $L^p$, for $p \neq 4, 6, \ldots$, see \cite{Henson2002}*{Example~3.18} and Lusky~\cite{Lusky1978}. Finally, it is an unpublished result of Henson that the $L^p$ Banach spaces are model-complete for all $p \geq 1$. We are grateful to Ward Henson for explaining to us the subtleties of the model theory of the $L^p$ spaces and providing the references. See also \cite{BenYaacov2011} for more details.

Thus Corollary~\ref{c:tau-qf} implies the following.
\begin{cor}[\cite{Cuth2019p}]
  \label{c:CDDK}
  The isometry classes of the Gurarij space and $L^p$ for $p \geq 1$ are $G_\delta$ sets in $(\Mod(T_0), \tauqf)$.
\end{cor}

\begin{remark}
  \label{rem:Lp}
  The results in \cite{Cuth2019p} are more detailed. For example, they show in addition that these isomorphism classes are $G_\delta$-complete.
\end{remark}


\section{The isomorphism equivalence relation: the general case}
\label{sec:isomorphism}

Recall that if $\Gamma$ is a pointclass, a Borel equivalence relation $E$ on a standard Borel space $X$ is called \df{potentially $\Gamma$} if there exists a Polish topology $\tau$ on $X$ compatible with the Borel structure such that $E$ is in $\Gamma$ with respect to the topology $\tau \times \tau$.

The main results of the next two sections are generalizations to the metric setting of two theorems of Hjorth and Kechris~\cite{Hjorth1996} that characterize smooth and essentially countable isomorphism relations. The first one works for arbitrary metric structures and is just a combination of the characterizations of atomic and $\aleph_0$-categorical structures from the previous section and the well-known fact that an equivalence relation is smooth iff it is potentially $\bPi^0_2$ iff it is potentially closed.

We start by recalling a consequence of the metric version of the López-Escobar theorem from \cite{BenYaacov2017a}: if $X \sub \Mod(L)$ is Borel and invariant under isomorphism, then there exists a sentence $\phi \in \Lomo(L)$ such that $p \in X$ iff $\phi^p=0$. 

\begin{theorem}
	\label{th:AtomicSmooth}
    Let $T$ be a countable $\Lomo(L)$-theory and denote by $\cong_T$ the equivalence relation of isomorphism restricted to $\Mod(T)$. Then the following are equivalent:
	\begin{enumerate}
   		\item $\cong_T$ is smooth;
		\item There exists a fragment $\cF$ such that for every $\xi \in \Mod(T)$, the theory $\Th_{\cF}(M_\xi)$ is $\aleph_0$-categorical;
		\item There exists a fragment $\cF$ such that for every $\xi \in \Mod(T)$, $M_\xi$ is $\cF$-atomic.
	\end{enumerate}
\end{theorem}
\begin{proof}
  \begin{cycprf}
  \item[\impnext] Let $Y$ be a Polish space and let $f \colon \Mod(T) \to Y$ be a Borel mapping such that $\xi \cong \eta$ if and only if $f(\xi) = f(\eta)$. Let $\{U_n\}_{n \in \N}$ be a countable basis for $Y$. Then
\begin{equation*}
  M \cong N \iff \forall n\  \big( M \in f^{-1}(U_n) \Leftrightarrow N \in f^{-1}(U_n) \big).  
\end{equation*}
	Now, each $f^{-1}(U_n)$ is an invariant Borel set, so by the López-Escobar theorem cited above, for each $n \in \N$, there are $\Lomo(L)$-sentences $\phi_n$ and $\psi_n$ such that $f^{-1}(U_n)=\Mod(\phi_n = 0)$ and $f^{-1}(X \setminus U_n)=\Mod(\psi_n = 0)$. Let $\cF$ be the fragment generated by $\{\phi_0, \psi_0,\ldots \}$. Then in the Polish topology $t_\cF$, each isomorphism class is closed. Corollary~\ref{c:CategoricalClosed} then implies that for each $\xi \in \Mod(T)$, $\Th_\cF(M_\xi)$ is $\aleph_0$-categorical.
	
  \item[\impnext] This follows from Corollary~\ref{c:CategoricalClosed} and Theorem~\ref{th:AtomicG_delta}.
    
  \item[\impfirst] By Theorem \ref{th:AtomicG_delta}, for any $\xi \in \Mod(T)$, $[\xi]$ is $G_\delta$ in the topology $t_{\cF}$ . Then \cite{Gao2009a}*{Theorem 6.4.4} implies that $\cong_\phi$ is smooth. Alternatively, $\Th_{\cF}(M_\xi)$ is a complete isomorphism invariant for $\xi$.
  \end{cycprf}
\end{proof}

Next we aim to characterize $\bSigma^0_2$ isomorphism classes. We will make use of this in the next section.

For the proof of the next lemma, we will need Vaught transforms in the space $\Mod(T)$ as developed in \cite{BenYaacov2017a}. If $M$ is a separable structure, denote by 
\begin{equation*}
  D(M) = \set[\big]{\bar y \in M^\N : \set{y_i : i \in \N} \text{ is dense in } M}.
\end{equation*}
$D(M)$ is clearly a $G_\delta$ set in $M^\N$, and therefore a Polish space. If $\xi \in \Mod(T)$, denote by $\pi \colon D(M_\xi) \to [\xi]$ the surjective map given by
\begin{equation}
  \label{eq:def-pi}
  \phi^{\pi(y)}(u) = \phi^{M_\xi}(y(u_0), \ldots, y(u_{n-1})) \quad \text{for all } \phi \in \cF, u \in \N^n, y \in D(M_\xi).
\end{equation}
In order to describe open sets in $D(M)$, it will be convenient to have a pseudo-metric defined on tuples of elements of $M$ of different length. For $m, n \leq \omega$ with $\min(m, n) < \omega$ and $\bar a \in M^m, \bar b \in M^n$, we define:
\begin{equation*}
  d(\bar a, \bar b) = \max \set{d(a_i, b_i) : i < \min(m, n)}.
\end{equation*}
For $M \models T$, $r > 0$, and $u \in \N^{<\N}$, let $B^{D(M)}_r(u)= \{ y \in D(M): d(y,u) < r \}$.

The quantifiers $\exists^*$ and $\forall^*$ mean as usual ``for non-meagerly many'' and ``for comeagerly many'', respectively.
For $A \subseteq \Mod(T)$, $u \in \N^{<\N}$, and $k>0$, we define the sets $A^\triangle$ and $A^{\triangle u,k}$ by
\begin{align*}
\xi \in A^\triangle &\iff  \exists^* y \in D(M_\xi) \quad \pi(y) \in A; \\
\xi \in A^{\triangle u,k} &\iff  \exists^* y \in B^{D(M)}_{1/k}(u) \quad \pi(y) \in A.
\end{align*}
For a Polish space $X$ and a Baire measurable function $f \colon X \to \R$, we define $\infstar_{x \in X} f(x)$ by
\begin{equation*}
  \infstar_{x \in X} f(x) < s \iff \exists^* x \in X \quad f(x) < s, \quad \text{ for } s \in \R.
\end{equation*}
For a Borel subset $A \sub \Mod(T)$ and $k \in \N$, define the function $U_A^{*k} \colon \Mod(T) \times \N^k \to \R$ by
  \begin{equation*}
    U_A^{*k}(\xi, u) = \infstar_{y \in D(M_\xi)} \chi_{A}(\pi(y)) \vee kd(y,u).
  \end{equation*}
  Here $\chi_A$ denotes the characteristic function of $A$. Note also that
  \begin{equation}
    \label{eq:property:U}
  \big( \forall^* y \in B_{1/k}^{D(M_\xi)}(u) \ \pi(y) \in A \big) \iff U^{*k}_A(\xi, u) \geq 1.
\end{equation}
The main significance of the function $U^{*k}_A$ is that it can be captured by a formula. More precisely, the following holds.
  \begin{theorem}[\cite{BenYaacov2017a}*{Theorem~6.3}]
    \label{th:LopezEscobar}
    Let $T$ be a countable $\Lomo$ theory, let $A \sub \Mod(T)$ be a Borel subset, and let $k \in \N$. Then there exists a $\Lomo$ formula $\phi_{A, k}(\bar x)$ with $k$ free variables such that
    \begin{equation*}
      U^{*k}_A(\xi, u) = \phi_{A, k}^\xi(u), \quad \text{ for all } \xi \in \Mod(T), u \in \N^k.
    \end{equation*}
  \end{theorem}
  The statement of \cite{BenYaacov2017a}*{Theorem~6.3} uses a slightly different metric on tuples from our $d$ for the definition of $U^{*k}_A$ but the proof is still valid.

  We are finally ready to state our next lemma, which allows us to represent invariant $\bSigma^0_2$ sets for an arbitrary Polish topology on $\Mod(T)$ as $\bSigma^0_2$ sets for a topology of the form $t_\cF$ for a fragment $\cF$ in a uniform way.
\begin{lemma}
	\label{l:Sigma02}
	Let $\cF$ be a fragment and let $T$ be a countable $\cF$-theory. Let $t$ be a Polish topology on $\Mod(T)$ whose open sets are Borel subsets of $\Mod(T)$. Then there exists a fragment $\cF' \supseteq \cF$ such that for every $\bSigma^0_2(t)$-set $A \sub \Mod(T)$, every $u \in \N^{<\N}$, and $k>0$, we have that $A^\triangle, A^{\triangle u,k} \in \bSigma^0_2(t_\cF)$. In particular, if $\cong_T$ is potentially $\bSigma^0_2$, then there exists a fragment $\cF' \supseteq \cF$ such that every isomorphism class is $\bSigma^0_2(t_{\cF'})$.
\end{lemma}
\begin{proof}
  Let $\cB$ be a countable basis of closed sets for the topology $t$, so that every $t$-closed set is an intersection of elements of $\cB$. Let $\cF'$ be the fragment generated by $\cF$ and the formulas $\set{\phi_{B, k} : B \in \cB, k \in \N}$ as given by Theorem~\ref{th:LopezEscobar}. Let now $A \in \bSigma^0_2(t)$ be arbitrary and let $A_{n,m}$ for $n, m \in \N$ be such that $A_{n,m}\in \cB$ and $A=\bigcup_n \bigcap_m A_{n,m}$.
	Then we have:
    \begin{equation*}
      \begin{split}
        \xi \in A^{\triangle u,k} \iff&  \exists^* y \in B^{D(M_\xi)}_{1/k}(u) \ \exists n \ \forall m \quad \pi(y)\in A_{n,m} \\
        \iff& \exists n \ \exists^* y \in B^{D(M_\xi)}_{1/k}(u) \ \forall m \quad \pi(y) \in A_{n,m} \\
        \iff& \exists n \ \exists u' \in \N^{<\N} \ \Big( d^{M_\xi}(u', u) \leq 1/k \And \\
          &\qquad \exists k' \in \N \ \forall^* y \in B^{D(M_\xi)}_{1/k'}(u') \ \forall m \quad  \pi(y) \in A_{n,m} \Big) \\
          \iff& \exists n \ \exists u' \in \N^{<\N} \ \Big( d^{M_\xi}(u',u) \leq 1/k \And
            \exists k' \ \forall m \ \phi_{A_{n, m}, k'}^\xi(u') \geq 1 \Big).
      \end{split}
    \end{equation*}
As both sets $\set{\xi : d^{\xi}(u', u) \leq 1/k}$ and $\set{\xi :  \phi_{A_{n, m}, k'}^\xi(u') \geq 1}$ are $t_{\cF'}$-closed, we get that $A^{\triangle u, k} \in \bSigma^0_2(t_{\cF'})$.
\end{proof}

The following definition is important for characterizing $\bSigma^0_2$ isomorphism classes.
\begin{defn}
  \label{df:aleph0-rigid}
  Let $\cF$ be a fragment and let $T$ be an $\cF$-theory.
  We will say that a type $p \in \tS_k(T)$ is \df{rigid} if whenever $(M, \bar a)$ and $(N, \bar b)$ are two realizations of $p$ with $M$ and $N$ separable, then $M \cong N$.
\end{defn}

\begin{prop}
  \label{p:Sigma02-implies-rigid}
  Let $\cF$ be a fragment, let $T$ be an $\cF$-theory, and let $\xi \in \Mod(T)$. Suppose that $[\xi]$ is a $\bSigma^0_2$ set in $t_\cF$. Then there exists $k > 0$ such that the set
  \begin{equation*}
	\set{\bar a \in M_\xi^k : \tp_\cF(\bar a) \text{ is rigid}}
  \end{equation*}
	has non-empty interior in $M_\xi^k$.
\end{prop}
\begin{proof}
  Let $[\xi] = \bigcup_n A_n$, where each $A_n$ is a closed set in $t_\cF$. Write $M = M_\xi$ and let $\pi \colon D(M) \to [\xi]$ be the map defined by \eqref{eq:def-pi}. By the Baire category theorem, there exists $n_0 \in \N$ such that $\pi^{-1}(A_{n_0})$ has non-empty interior. For brevity, put $A = A_{n_0}$. We may assume that $\pi^{-1}(A)$ contains an open set $U$ of the form
\begin{equation*}
  U = \set{y \in D(M) : d(y, \bar a_0) < r}
\end{equation*}
for some $k \in \N$ and $\bar a_0 \in M^k$. We claim that for every $\bar a \in B_r(\bar a_0)$, $\tp_\cF \bar a$ is rigid.

Indeed, fix such an $\bar a$ and let $(N, \bar b) \equiv_\cF (M, \bar a)$ with $N$ separable. We will find an enumeration $z \in D(N)$ such that $\pi(z) \in \cl{\pi(U)} \sub A \sub [\xi]$, which will imply that $N \cong M$. Choose $z \in D(N)$ arbitrary such that $z|_k = \bar b$. Now given $n \in \N$, a formula $\phi(x_0, \ldots, x_{n-1}) \in \cF$, and $\eps > 0$, we need to find $y \in D(M)$ such that $d(y, a_0) < r$ and $|\phi^M(y|_n) - \phi^N(z|_n)| < \eps$. We may assume that $n \geq k$, $\phi \geq 0$ and $\phi^N(z|_n) = 0$. We have $\big(\inf_{\bar x} \phi(\bar b, x_k, \ldots, x_{n-1}) \big)^N = 0$. As $(M, \bar a) \equiv_\cF (N, \bar b)$, this implies that there exists $\bar e \in M^{n-k}$ such that $\phi^M(\bar a \bar e) < \eps$. Now it is enough to take $y|_n = \bar a \bar e$ and prolong it arbitrarily.
\end{proof}
  
We finish with a lemma that says that the collection of rigid types is not too complicated.
\begin{lemma}
  \label{l:rigid-pi11}
  Let $\cF$ be a fragment, let $T$ be an $\cF$-theory and suppose that the equivalence relation $\cong_T$ is Borel. Then for every $k \in \N$, the set
  \begin{equation*}
    \set{p \in \tS_k(T) : p \text{ is rigid}}
  \end{equation*}
  is $\bPi^1_1$.
\end{lemma}
\begin{proof}
  Let $u = (0, \ldots, k-1)$ and note that for $p \in \tS_k(T)$,
  \begin{equation*}
    p \text{ is rigid} \iff \forall \xi, \eta \in \Mod(T) \quad \tp^{M_\xi} u = \tp^{M_\eta} u = p \implies M_\xi \cong M_\eta
  \end{equation*}
  and isomorphism being Borel, the latter condition is clearly $\bPi^1_1$.
\end{proof}


\section{The isomorphism equivalence relation: locally compact structures}
\label{sec:locally-compact}

The following is our main theorem about the complexity of isomorphism of locally compact structures and this section is devoted to its proof.
\begin{theorem}
  \label{th:loc-cpct-HK}
  Let $T$ be a theory such that all of its separable models are locally compact. Then the following are equivalent:
	\begin{enumerate}
		\item \label{i:th:HK:1} $\cong_T$ is potentially $\bSigma^0_2$;
		
		\item \label{i:th:HK:2} There exists a fragment $\cF$ such that $T$ is an $\cF$-theory and for every $\xi \in \Mod(T)$, there is $k \in \N$ such that the set
		\begin{equation*}
		\set{\bar a \in M_\xi^k : \tp_\cF \bar a \text{ is rigid}}
		\end{equation*}
		has non-empty interior in $M_\xi^k$;
		
		\item \label{i:th:HK:3} $\cong_T$ is essentially countable.
	\end{enumerate} 
\end{theorem}

If $M$ is a structure and $\cF$ is a fragment, let
\begin{equation*}
  \Theta_n^\cF (M) = \set{p \in \tS_n(\Th_\cF(M)) : M \text{ realizes } p}.
\end{equation*}
If $\cF$ is understood, we will often omit it from the notation.

If $(Z, d)$ is a metric space, $z_0 \in Z$ and $r > 0$, we denote by
\begin{equation*}
  B_r(z_0) = \set{z \in Z : d(z, z_0) < r} \quad \And \quad B'_r(z_0) = \set{z \in Z : d(z, z_0) \leq r}
\end{equation*}
the open and closed balls with center $z_0$ and radius $r$, respectively.
If $Z$ is in addition locally compact, denote
\begin{equation*}
\rho(z_0) = \sup \set{r \in \R : \cl{B_r(z_0)} \text{ is compact}} = \sup \set{r \in \R : B_r'(z_0) \text{ is compact}}.
\end{equation*}
If $\xi \in \Mod(T)$ for some theory $T$, we will often write $\rho_\xi$ instead of $\rho_{M_\xi^k}$ for some Cartesian power $k$ that is understood from the context.

The next lemma collects some basic facts about type spaces of theories with locally compact models.
\begin{lemma}
  \label{l:loc-cpct-type-space}
  Let $\cF$ be a fragment, let $T$ be an $\cF$-theory, and let $M$ be a separable, locally compact model of $T$. Let $\Phi \colon (M^n, d) \to (\tS_n(T), \partial)$ be defined by $\Phi(\bar a) = \tp_\cF \bar a$. Then the following hold:
	\begin{enumerate}
    \item \label{i:p:loc-cpct-type-space:1} $\Phi$ is a contraction for the metrics $d$ on $M^n$ and $\dtp$ on $\tS_n(T)$.
    \item \label{i:p:loc-cpct-type-space:2} If $\bar a \in M^n$ and $r < \rho(\bar a)$, then $\Phi(B_r'(\bar a)) = B_r'(\Phi(a))$. In particular, $B_r'(\tp \bar a) \sub \Theta_n(M)$ and $B_r'(\tp \bar a)$ is $\dtp$-compact.
    \item \label{i:p:loc-cpct-type-space:3} If $B_r(\bar a)$ is an open ball with $r \leq \rho(\bar a)$, then $\Phi(B_r(\bar a)) = B_r(\Phi(\bar a))$. In particular, $\Phi$ is an open map $M^n \to (\tS_n(T), \dtp)$.
    \item \label{i:p:loc-cpct-type-space:4} The set $\Theta_n(M)$ is open in $(\tS_n(T), \dtp)$ and the space $(\Theta_n(M), \dtp)$ is locally compact and separable. 
	\end{enumerate}
\end{lemma}
\begin{proof}
  \ref{i:p:loc-cpct-type-space:1} This is clear.
  
  \ref{i:p:loc-cpct-type-space:2} Let $p \in \tS_n(T)$ be such that $\partial(p, \tp \bar a) \leq r$. Then by the definition \eqref{eq:dtp1} of $\dtp$, there is a sequence $(\bar b_i)_{i \in \N}$ of elements of $B'_r(\bar a)$ such that for every $\cF$-formula $\phi$, $\phi^M(\bar b_i) \to \phi(p)$ and $\limsup_{i \to \infty} d(\bar b_i, \bar a) \leq r$. Let $\bar b$ be a limit point of the $\bar b_i$. Then $\bar b \models p$ and $d(\bar b, \bar a) \leq r$, as required.

  \ref{i:p:loc-cpct-type-space:3} Let $B_r(\bar a)$ be an open ball around $\bar a$ with $r < \rho(\bar a)$. Using \ref{i:p:loc-cpct-type-space:2}, we have that
  \begin{equation*}
    \Phi(B_r(\bar a)) = \Phi(\bigcup_{r' < r} B'_{r'}(\bar a)) = \bigcup_{r' < r} \Phi(B'_{r'}(\bar a)) = \bigcup_{r' < r} B'_{r'}(\Phi(\bar a)) = B_r(\Phi(\bar a)).
  \end{equation*}
  We conclude by observing that, as $M$ is locally compact, the sets $\set{B_r(\bar a) : \bar a \in M^n, r < \rho(\bar a)}$ form a basis for the topology of $M^n$.
	
	\ref{i:p:loc-cpct-type-space:4} This follows from \ref{i:p:loc-cpct-type-space:3} and the fact that the open, continuous image of a locally compact space is locally compact.
\end{proof}

For a fixed $\tau$-open set $U \sub \tS_n(T)$ and $\eps > 0$, define the following equivalence relation $R_{U, \eps}(M)$ on $U \cap \Theta_n(M)$:
\begin{equation*}
p \eqrel{R_{U, \eps}} q \iff \exists p_0, \ldots, p_k \in U \ p_0 = p \And p_k = q \And \forall i < k \ \partial(p_i, p_{i+1}) < \eps.
\end{equation*}
Note that each $R_{U, \eps}$-class is $\partial$-open, so by Lemma~\ref{l:loc-cpct-type-space}, there are only countably many of them.

\begin{lemma}
	\label{l:compact-class}
	Let $M$ be locally compact. Then for every $p \in \Theta_n(M)$, there exist a basic $\tau$-open $U$ and $\eps > 0$ such that $\cl[\partial]{[p]_{R_{U, \eps}}}$ is $\dtp$-compact and contained in $\Theta_n(M)$. In particular, $\cl[\tau]{[p]_{R_{U, \eps}}}$ is $\tau$-compact and $\cl[\tau]{[p]_{R_{U, \eps}}} \sub \Theta_n(M)$.
\end{lemma}
\begin{proof}
  Use Lemma~\ref{l:loc-cpct-type-space} to find $\eps > 0$ such that $B'_{2\eps}(p)$ is $\dtp$-compact and contained in $\Theta_n(M)$. Then the $\tau$-topology and the $\partial$-topology coincide on $B_{2\eps}(p)$ and therefore there exists a basic $\tau$-open $U \sub \tS_n(T)$ such that $p \in U$ and $U \cap B_{2\eps}(p) \sub B_\eps(p)$. This implies that $[p]_{R_{U, \eps}} \sub B_\eps(p)$. As $B'_\eps(p)$ is $\tau$-closed, we obtain that $\cl[\tau]{[p]_{R_{U, \eps}}} \sub B'_\eps(p) \sub \Theta_n(M)$. Furthermore, $\cl[\tau]{[p]_{R_{U, \eps}}}$, being a $\dtp$-closed subset of the $\dtp$-compact set $B'_\eps(p)$, is $\dtp$-compact and therefore also $\tau$-compact.
\end{proof}

\begin{lemma}
  \label{l:p-to-pRUe-Borel}
  Let $T$ be a theory whose models are locally compact and let $k \in \N$. Then the following maps are Borel:
  \begin{enumerate}
  \item \label{i:p:Borel:1} $\Mod(T) \times \N^k \to \R$, $(\xi, u) \mapsto \rho_\xi(u)$;
  \item \label{i:p:Borel:2} $\Mod(T) \times \N^k \times \R^+ \to K(\tS_k(T))$, $(\xi, u, r) \mapsto B'_r(\tp^\xi u)$ if $r < \rho_\xi(u)$ and $\emptyset$, otherwise. Here $K(\tS_k(T))$ denotes the collection of $\tau$-compact subsets of $\tS_k(T)$ equipped with the Vietoris topology.
  \end{enumerate}
\end{lemma}
\begin{proof}
  \ref{i:p:Borel:1} We consider for simplicity of notation the case $k = 1$. For $r \in \R^+$, we have:
  \begin{equation*}
    \begin{split}
      \rho_{\xi}(u) > r \iff& \exists r' > r \  B_{r'}^{M_\xi}(u) \text{ is totally bounded} \\
      \iff& \exists r' > r \ \forall \eps > 0 \ \exists v_0, \ldots, v_{m-1} \in \N \ \forall w \in \N  \\
       &\qquad d^\xi(u, w) < r' \implies \exists i < m \ d^\xi(v_i, w) < \eps
    \end{split}
  \end{equation*}
  and this is clearly Borel. (The quantifiers on $r'$ and $\eps$ can be taken over the rationals.)
  
  \ref{i:p:Borel:2} We need to check that for any basic $\tau$-open $U \sub \tS_k(T)$, the set
  \begin{equation*}
	W = \set{(\xi, u, r) \in \Mod(T) \times \N^k \times \R^+ : r < \rho_\xi(u) \And B_r'(\tp^\xi u) \cap U \neq \emptyset}
  \end{equation*}
  is Borel. Let $U = \oset{\phi < s}$ for some formula $\phi$. We have that:
  \begin{equation*}
    \begin{split}
      (\xi, u, r) \in W &\iff r < \rho_\xi(u) \And \exists s' < s \ \forall r' > r \ B_{r'}(\tp^\xi u) \cap \oset{\phi < s'} \neq \emptyset \\
      &\iff r < \rho_\xi(u) \And \exists s' < s \ \forall r' > r \ \exists v \in \N^k \\
      & \qquad \qquad d^\xi(u, v) < r' \And \phi^\xi(v) < s',
    \end{split}
  \end{equation*}
  which is clearly Borel. The left to right direction of the first equivalence is clear. To go from right to left, suppose that the right-hand side holds. For $n \in \N$, let $p_n \in B_{r+1/n}(\tp^\xi u)$ be such that $\phi(p_n) < s'$. It follows from Lemma~\ref{l:loc-cpct-type-space} \ref{i:p:loc-cpct-type-space:2} that $\rho(\tp^\xi u) \geq \rho_\xi(u) > r$, so we may assume that the sequence $p_n$ $\dtp$-converges to some $p$. Then $\phi(p) \leq s' < s$ and $\dtp(p, \tp^\xi u) \leq r$, so $p \in B_r'(\tp^\xi u) \cap U$. The second equivalence follows from Lemma~\ref{l:loc-cpct-type-space} \ref{i:p:loc-cpct-type-space:3}.
\end{proof}

\begin{lemma}
  \label{l:ii-implies-Borel}
  Suppose that $T$ is a theory satisfying item \ref{i:p:loc-cpct-type-space:2} of the statement of Theorem~\ref{th:loc-cpct-HK}. Then $\cong_T$ is Borel.
\end{lemma}
\begin{proof}
  Fix a fragment $\cF$ satisfying the condition in item \ref{i:th:HK:2} of Theorem~\ref{th:loc-cpct-HK}. We will show that
  \begin{equation}
    \label{eq:isom-Borel}
    M_\xi \cong M_\eta \iff \forall k \ \forall u \in \N^k \ \forall r < \rho_\xi(u) \ \exists v \in \N^k  \quad \tp^{\eta} v \in B_r'(\tp^{\xi} u),
  \end{equation}
  where the types are taken with respect to $\cF$.
  Suppose first that $M_\xi \cong M_\eta$ and let $f \colon M_\xi \to M_\eta$ be an isomorphism. Let $k$, $u$ and $r$ be given. Then any $v \in f(B_r(u))$ works because $\tp^{\eta} v = \tp^{\xi} f^{-1}(v)$.

  Now suppose that the right-hand side of \eqref{eq:isom-Borel} holds. Let $k$ be such that
  \begin{equation*}
    \set{\bar a \in M_\xi^k : \tp \bar a \text{ is rigid}}
  \end{equation*}
  has non-empty interior in $M_\xi^k$. Let $u \in \N^k$ and $r < \rho_\xi(u)$ be such that for all $\bar a \in B_r'(u)$, $\tp^{\xi} \bar a$ is rigid. It follows from Lemma~\ref{l:loc-cpct-type-space} that every $p \in B_r'(\tp^{\xi} u)$ is realized in $M_\xi$ and is rigid. The hypothesis implies that some $p \in B_r'(\tp^{\xi} u)$ is realized in $M_\eta$. Thus $M_\xi$ and $M_\eta$ realize a common rigid type, so they must be isomorphic. This concludes the proof of \eqref{eq:isom-Borel}.

  Finally, it follows from Lemma~\ref{l:p-to-pRUe-Borel} that the condition in the right-hand side of \eqref{eq:isom-Borel} is Borel.
\end{proof}

\begin{proof}[Proof of Theorem~\ref{th:loc-cpct-HK}]
	\begin{cycprf}
		\item[\impnext] This follows from Lemma~\ref{l:Sigma02} and Proposition~\ref{p:Sigma02-implies-rigid}.
		\item[\impnext] By a well-known result of Kechris (see \cite{Hjorth2000a}*{Lemma~5.2}), in order to prove that $\cong_T$ is essentially countable, it suffices to produce a standard Borel space $Y$ and a Borel map $\Psi \colon \Mod(T) \to Y$ such that the image of each isomorphism class is countable and the images of different classes are disjoint. Let $Y = \bigsqcup_k K(\tS_k(T))$ and define $\Psi$ as follows: for a given $\xi \in \Mod(T)$, choose $k \in \N$ and $u \in \N^k$ such that $p = \tp^\xi u$ is rigid, choose a basic $\tau$-open $U \sub \tS_k(T)$ and a rational $\eps > 0$ such that Lemma~\ref{l:compact-class} holds for $p$ and set $\Psi(\xi) = \cl[\tau]{[p]_{R_{U, \eps}}}$. We check that this can be done in a Borel way. First, by Lemma~\ref{l:ii-implies-Borel}, the equivalence relation $\cong_T$ is Borel. Now Lemma~\ref{l:rigid-pi11} implies that the set
\begin{equation*}
  W = \set{(\xi, u) \in \Mod(T) \times \N^{<\N} : \tp^\xi u \text{ is rigid}}
\end{equation*}
is $\bPi^1_1$ and by assumption, each section $W_\xi$ for $\xi \in \Mod(T)$ is non-empty. Then by number uniformization, there exists a Borel map $\Phi \colon \Mod(T) \to \N^{<\N}$ such that $\tp^\xi \Phi(\xi)$ is rigid. Write $u = \Phi(\xi)$, $k = |u|$ and $p = \tp^\xi u$. Next, for Lemma~\ref{l:compact-class} to hold, we need that $2\eps < \rho_\xi(u)$ and $U \cap B'_{2\eps}(p) \sub B'_\eps(p)$. Thus $\eps$ and $U$ can also be chosen in a Borel way by Lemma~\ref{l:p-to-pRUe-Borel}. Finally, note that $\cl[\tau]{[p]_{R_{U, \eps}}} = \cl[\tau]{U \cap B_\eps(p)}$ and this is again Borel. Indeed, for a $\tau$-open $V \sub \tS_{k}(T)$, we have
\begin{equation*}
  \begin{split}
    \cl[\tau]{U \cap B_\eps(p)} \cap V \neq \emptyset &\iff B_\eps(p) \cap U \cap V \neq \emptyset \\
    &\iff \exists v \in \N^{k} \ d^\xi(u, v) < \eps \And \tp^\xi v \in U \cap V
  \end{split}
\end{equation*}
and this is a Borel condition.

	As there are only countably many choices for $k$, $U$ and $\eps$, $R_{U, \eps}$ has only countably many classes, and for isomorphic $M$ and $N$, $\Theta_k(M) = \Theta_k(N)$, we obtain that the image of each isomorphism class is countable. Suppose now that $\Psi(\xi) = \Psi(\eta)$ for some $\xi, \eta \in \Mod(T)$. This implies that there exists $k \in \N$ and $\bar a \in M_\xi^k$ such that $p = \tp \bar a$ is rigid and $p \in \Psi(\xi) = \Psi(\eta) \sub \Theta(M_\eta)$. In particular, $M_\eta$ realizes $p$ and by rigidity, we must have $M_\xi \cong M_\eta$.
		\item[\impfirst] This is obvious. 
	\end{cycprf}
  \end{proof}

  We conclude this section with two remarks comparing Theorem~\ref{th:loc-cpct-HK} to the analogous result of Hjorth and Kechris about discrete structures and showing with examples that natural modifications of Theorem~\ref{th:loc-cpct-HK} fail.
\begin{remark}
  \label{rem:cmpHK1}
  One may ask whether it is possible to replace the notion of a rigid type in Theorem~\ref{th:loc-cpct-HK} with the requirement that the type (considered as a theory) has an atomic model. (Indeed, this is the condition used by Hjorth and Kechris in the discrete case.) We first note that in classical logic, every rigid type admits an atomic model. This is simply because if $p$ is rigid, then $\tS_n(p)$ is countable for every $n$ and in every countable Polish space, isolated points are dense. However, this fails in continuous logic as can be seen from the following example. Let $\bG$ denote the Gurarij Banach space. This is the Fraïssé limit of finite-dimensional Banach spaces and it is $\aleph_0$-categorical in the $\Loo$ fragment. Then there exists a four-dimensional subspace $E \sub \bG$ such that the theory of $\bG$ with parameters for $E$ does not admit an atomic model (see Ben Yaacov--Henson~\cite{BenYaacov2017b}*{Example~6.6}). For us, this means that if $\bar e \in \bG^4$ is a basis for $E$, then $\tp \bar e$ does not admit an atomic model. However, this type is rigid because $\Th(\bG)$ is $\aleph_0$-categorical. We do not know a similar example for locally compact structures but strongly suspect that one exists (allowing arbitrary fragments). 
\end{remark}

\begin{remark}
  \label{rem:one-gen-Cstar-algebras}
  Another easy consequence of the Hjorth--Kechris result is that isomorphism of finitely generated discrete structures is an essentially countable equivalence relation. This also fails in the continuous setting as can be seen by combining several results from the literature as follows. Thiel and Winter \cite{Thiel2014}*{Theorem~3.8} have proved that separable, $\cZ$-stable \Cstar-algebras are singly generated and Toms and Winter \cite{Toms2008}*{Theorem~2.3} have shown that approximately divisible, separable \Cstar-algebras are $\cZ$-stable. It follows from the proof of the main theorem in Elliott~\cite{Elliott1993} that separable, simple, AI algebras are approximately divisible, and finally, Sabok~\cite{Sabok2016} has proved that the isomorphism relation for separable, simple, AI algebras is bi-reducible with the universal equivalence relation given by a Polish group action. By combining all of this, we conclude that isomorphism for singly generated \Cstar-algebras is universal for orbit equivalence relations of Polish group actions.
\end{remark}


\section{Pseudo-connected metric structures and a theorem of Kechris}
\label{sec:pseudo-connected-metric-struct}

In this section, we use some basic model theory and Theorem~\ref{th:loc-cpct-HK} to deduce a generalization of a theorem of Hjorth about pseudo-connected locally compact metric spaces. We also show how to apply this to recover a theorem of Kechris about orbit equivalence relations of actions of locally compact Polish groups.

Let $(Z, d)$ be a locally compact metric space. We recall from \cite{Gao2003} that the \df{pseudo-component} of a point $z_0 \in Z$ is defined as:
\begin{multline}
  \label{eq:pseudo-comp}
  C[z_0] \coloneqq \set{z \in Z : \exists z_1, \ldots, z_n \ \forall i = 0, \ldots, n-1 \\
    d(z_i, z_{i+1}) < \rho(z_i) \And z_n = z}.
\end{multline}
The space $Z$ is called \df{pseudo-connected} if it has only one pseudo-component. Examples of pseudo-connected locally compact metric spaces are the connected spaces and the \df{proper} metric spaces (spaces where every closed ball is compact).

\begin{remark}
  \label{rem:unbounded}
  Our model-theoretic setting is limited to bounded structures, so, strictly speaking, the example of proper metric spaces does not fall into our framework. It is possible to treat unbounded structures in several ways. One is to replace the distance predicate $d$ by infinitely many predicates $\set{d_n}_{n \in \N}$ defined by $d_n = d \wedge n$ (and take $d_1$ as the ``official'' distance used to define the moduli of continuity). Then one needs to add the axiom 
  \begin{equation*}
    \inf_{x, y} \bigvee_n 1 \wedge (n - d_n(x, y)) = 1
  \end{equation*}
  which states that $d(x, y)$ is finite for all points $x, y$.

  Another possibility is to replace the metric $d$ by $d' = d/(1+d)$ which is bounded by $1$. Then the condition of being proper is replaced by the condition that every closed ball of radius less than $1$ is compact and the new space is still pseudo-connected. Note that both encodings preserve the equivalence relation of isomorphism.
\end{remark}

Recall that if $M$ is a model and $A \sub M$, the \df{algebraic closure of $A$} is defined by
\begin{equation*}
  \acl A \coloneqq \bigcup \set{D \sub M : D \text{ is $\Loo$-definable from $A$ and compact}}.
\end{equation*}
See \cite{BenYaacov2008}*{Section~9} for the definition of definable sets and more details on the algebraic closure operator. We will only use the notion of algebraic closure for the finitary fragment $\Loo$, so that all results of \cite{BenYaacov2008} apply.

We have the following basic fact about pseudo-components.
\begin{prop}
  \label{p:pseudo-comp}
  Let $M$ be a locally compact metric structure and let $a \in M$. Then $C[a] \sub \acl a$.
\end{prop}
\begin{proof}
  Let $b \in M$ and $r < \rho(b)$. We will show that the compact ball $B'_r(b)$ is definable from $b$. We will apply \cite{BenYaacov2008}*{Proposition~9.19 (2)} to the predicate $d(x, b) \dotminus r$. If the condition (2) is not verified, there exists $\eps > 0$ such that for every $n$ there exists a point $b_n$ with $d(b_n, b) < r + 1/n$ and $d(b_n, B'_r(b)) \geq \eps$. As $r < \rho(b)$, we may assume that $b_n \to b'$ but then $d(b, b') \leq r$ and $d(b', B'_r(b)) \geq \eps$, contradiction.

  Now the conclusion of the proposition follows from the definition \eqref{eq:pseudo-comp} and the transitivity of the algebraic closure operator (see \cite{BenYaacov2008}*{Proposition~10.11 (2)}).
\end{proof}

\begin{prop}
  \label{p:pseudo-conn-aleph0-categorical}
  Let $\cF$ be a fragment and let $T$ be an $\cF$-theory such that all models of $T$ are pseudo-connected. Then for every model $M \models T$ and every $c \in M$, the theory $\Th_\cF(M, c)$ is $\aleph_0$-categorical.
\end{prop}
\begin{proof}
  Suppose that $(N, b) \equiv_\cF (M, c)$. In particular, $(N, b) \equiv_\Loo (M, c)$, so there exists a model $K$ and $\Loo$-elementary embeddings $f \colon M \to K$ and $g \colon N \to K$ with $f(b) = g(c)$. We have from Proposition~\ref{p:pseudo-comp} that $M = \acl c$ and similarly, as $N \models T$, $N$ is also pseudo-connected, so $N = \acl b$. By \cite{BenYaacov2008}*{Corollary~10.5}, we have that
  \begin{equation*}
    f(M) = \acl f(b) = \acl g(c) = g(N),
  \end{equation*}
  so $g^{-1} \circ f$ is an isomorphism $M \to N$ sending $b$ to $c$.
\end{proof}

\begin{cor}[cf. \cite{Gao2003}*{Theorem~5.7}]
  \label{c:pseudo-conn-point-smooth}
  Let $L$ be a signature, let $\cF$ be a fragment of $\Lomo(L)$ and let $T$ be an $\cF$-theory such that all models of $T$ are pseudo-connected. Let $c$ be a constant symbol that does not appear in $L$. Let $T'$ be the theory $T$ considered in the expanded language $L \cup \set{c}$. Then the equivalence relation of isomorphism on $\Mod(T')$ is smooth.
\end{cor}
\begin{proof}
  This follows from Proposition~\ref{p:pseudo-conn-aleph0-categorical} and Theorem~\ref{th:AtomicSmooth}.
\end{proof}


The next result is a generalization of a theorem of Hjorth (\cite{Gao2003}*{Theorem~7.1}).
\begin{cor}
  \label{c:esslly-ctble}
  Let $T$ be a theory such that all of its models are pseudo-connected. Then $\cong_T$ is essentially countable.
\end{cor}
\begin{proof}
	  This follows from Proposition~\ref{p:pseudo-conn-aleph0-categorical} and Theorem~\ref{th:loc-cpct-HK}.
\end{proof}

\begin{remark}
  \label{rem:fin-many-pseudo-components}
  Corollary~\ref{c:esslly-ctble} also holds if one replaces ``are pseudo-connected'' with ``have finitely many pseudo-components''. Indeed, if $C_1, \ldots, C_n$ are the pseudo-components of a structure $M$, then $\prod_i C_i$ is open in $M^n$ and for every $\bar a \in \prod_i C_i$, we have that $M = \acl \bar a$, so the same argument works.
\end{remark}

As another application of results from the previous section, we present a new proof of a theorem of Kechris from \cite{Kechris1992a}, stating that orbit equivalence relations induced by continuous actions of locally compact Polish groups are essentially countable. First, we need a way of encoding group actions as metric structures.

Let $Z$ be a compact metrizable space and let $G \leq \Homeo(Z)$ be a closed, CLI subgroup. (Recall that a Polish group is \df{CLI} if its left uniformity is complete iff its right uniformity is complete. All Polish locally compact groups are CLI.) Denote by $\alpha$ the action $G \actson Z$. We will reduce the orbit equivalence relation $E_\alpha$ to isomorphism of metric structures with underlying space $(G, d)$, where $d$ is a fixed, right-invariant, compatible metric on $G$.

Define the metric $d_u$ on $\Homeo(Z)$ by
\begin{equation}
\label{eq:sup-metric-G}
d_u(h_1, h_2) = \sup \set{\delta(h_1 \cdot z, h_2 \cdot z) : z \in Z},
\end{equation}
where $\delta$ is some fixed compatible metric on $Z$ bounded by $1$. Note that $d_u$ restricts to a compatible, right-invariant metric on $G$ and therefore the metrics $d_u$ and $d$ are uniformly equivalent.

Let $\set{z_i}_{i \in \N}$ be a dense sequence in $Z$. Let $\Delta$ be a modulus of continuity such that every unary predicate on $G$ which is $1$-Lipschitz with respect to $d_u$ respects $\Delta$ with respect to $d$. Let $L$ be the language consisting of the metric $d$ and the unary predicates $\set{P_i}_{i \in \N}$ respecting $\Delta$. Let $T$ be the theory consisting of the Scott sentence of the metric space $(G, d)$. For each $z \in Z$, define an $L$-structure $M(z)$ with universe $(G, d)$ and predicates defined on $G$ by
\begin{equation*}
P_i^z(h) = \delta(h \cdot z, z_i).
\end{equation*}
As each $P_i^z$ is $1$-Lipschitz with respect to $d_u$, it respects $\Delta$, so $M(z)$ is a valid $L$-structure which is also a model of $T$. Also, the predicates $P_i^z$ code $z$ uniquely: if $P_i^z(1_G) = P_i^{z'}(1_G)$ for all $i$, then $z = z'$.
\begin{prop}
	\label{p:reduction-to-isom}
	The map $Z \to \Mod(T)$ given by $z \mapsto M(z)$ is a Borel reduction from $E_\alpha$ to isomorphism of models of $T$. 
\end{prop}
\begin{proof}
	One easily checks that the map $G \to G$, $h \mapsto hg$ is an isomorphism $M(z) \to M(g \cdot z)$.
	
	Conversely, suppose that $f \colon M(z) \to M(z')$ is an isomorphism and let $h_0 = f(1_G)$. Then
	\begin{equation*}
	P_i^z(1_G) = P_i^{z'}(h_0) = P_i^{h_0 \cdot z'}(1_G) \quad \text{ for all } i,
	\end{equation*}
	whence $z = h_0 \cdot z'$.
\end{proof}

\begin{cor}[Kechris~\cite{Kechris1992a}]
	\label{c:Kechris}
	Let $G \actson X$ be a Borel action of a locally compact Polish group $G$ on a Polish space $X$. Then its orbit equivalence relation is essentially countable.
\end{cor}
\begin{proof}
  Let $F(G)$ denote the space of closed subsets of $G$. It carries a compact Polish topology with basic open sets of the form
  \begin{equation*}
    \set{F \in F(G) : F \cap U_1 \neq \emptyset, \ldots, F \cap U_n \neq \emptyset, F \cap K = \emptyset},
  \end{equation*}
  where $U_1, \ldots, U_n \sub G$ are open and $K \sub G$ is compact. There is a natural action $G \actson F(G)$ by left translation and it is well-known that the action $G \actson F(G)^\N$ is universal for Borel actions of $G$ \cite{Becker1996}*{Theorem~2.6.1}. Thus it suffices to prove that the orbit equivalence relation of this action is essentially countable.

  Let $d'$ be any proper, right-invariant, compatible metric on $G$ (see Struble~\cite{Struble1974}) and let $d = d'/(1+d')$, so that $d$ is right-invariant and $(G, d)$ is pseudo-connected and bounded. The action $G \actson F(G)^\N$ gives an embedding of $G$ as a closed subgroup of $\Homeo(F(G)^\N)$. Let the language $L$ and the theory $T$ be defined as in the discussion preceding Proposition~\ref{p:reduction-to-isom}. By Proposition~\ref{p:reduction-to-isom}, the orbit equivalence relation of the action $G \actson F(G)^\N$ is Borel reducible to isomorphism of models of $T$. As $T$ contains the Scott sentence of $(G, d)$, all models of $T$ are pseudo-connected, so we can apply Corollary~\ref{c:esslly-ctble} to deduce that $\cong_T$ is an essentially countable equivalence relation. This concludes the proof.
\end{proof}

\bibliography{cont-logic-eqrel}

\end{document}